\documentclass[12pt,reqno]{amsart}

\usepackage[all]{xypic}
\usepackage{enumerate}
\usepackage{hyperref}
\usepackage{a4}

\DeclareMathOperator*{\res}{res}
\DeclareMathOperator{\sign}{sign}
\DeclareMathOperator{\spec}{spec}
\DeclareMathOperator{\img}{img}
\DeclareMathOperator{\rk}{rk}
\DeclareMathOperator{\tr}{tr}
\DeclareMathOperator{\Li}{Li}
\DeclareMathOperator*{\FP}{FP}
\DeclareMathOperator{\Td}{Td}
\DeclareMathOperator{\ch}{ch}
\DeclareMathOperator{\Tch}{Tch}
\newcommand\grr{\textrm{gr}}
\newcommand\Hurw{\mathrm{Hurw}}
\newcommand\Riem{\mathrm{Riem}}
\newcommand\triv{\mathrm{triv}}
\newcommand\goe{\mathfrak g}
\newcommand\noe{\mathfrak n}
\newcommand\zoe{\mathfrak z}

\DeclareMathOperator{\GL}{GL}
\DeclareMathOperator{\OO}{O}
\DeclareMathOperator{\Aut}{Aut}
\DeclareMathOperator{\eend}{end}

\theoremstyle{plain}
  \newtheorem{theorem}{Theorem}
  \newtheorem{corollary}{Corollary}
  \newtheorem{lemma}{Lemma}
  \newtheorem{proposition}{Proposition}
\theoremstyle{definition}

\theoremstyle{remark}
  \newtheorem{remark}{Remark}

\begin{document}

\title{On the eta invariant of (2,3,5) distributions}

\author{Stefan Haller}

\address{Stefan Haller,
         Department of Mathematics,
         University of Vienna,
         Oskar-Morgenstern-Platz 1,
         1090 Vienna,
         Austria.
	 \href{https://orcid.org/0000-0002-7064-2215}{ORCID 0000-0002-7064-2215}}

\email{stefan.haller@univie.ac.at}

\begin{abstract}
	We consider the Rumin complex associated with a generic rank two distribution on a closed 5-manifold.
	The Rumin differential in middle degrees gives rise to a self-adjoint differential operator of Heisenberg order two.
	We study the eta function and the eta invariant of said operator, twisted by unitary flat vector bundles.
	For (2,3,5) nilmanifolds this eta invariant vanishes but the eta function is nontrivial, in general.
	We establish a formula expressing the eta function of (2,3,5) nilmanifolds in terms of more elementary functions.
\end{abstract}




\maketitle

\section{Introduction}\label{S:intro}

	The eta invariant of the odd signature operator has been introduced by Atiyah, Patodi, and Singer \cite{APS73,APS75} 
	in order to extend the Hirzebruch signature theorem to compact manifolds with boundary.
	The eta invariant of a contact analogue of the odd signature operator has been studied by Biquard, Herzlich, and Rumin \cite{BHR07,R00}.
	The latter is a pseudohermitian invariant, that is, it depends on a contact form and an almost complex structure on the contact hyperplane field.
	In dimension three, and up to a local quantity, the contact eta invariant coincides \cite[Thm~1.1]{BHR07} with the $\nu$-invariant introduced by Biquard and Herzlich \cite{BH05} 
	which is a CR-invariant, i.e., independent of the contact form.
	More recently, Albin and Quan have shown that on contact manifolds of arbitrary dimension the relative contact eta invariant coincides with the Riemannian analogue \cite[Cor~2]{AQ22}.

	In this paper we study the eta invariant of generic rank two distributions on closed 5-manifolds.
	Recall that a generic rank two distribution in dimension five, a.k.a.\ (2,3,5) distribution, is a rank two subbundle $\mathcal D\subseteq TM$ in the tangent bundle of a smooth 5-manifold $M$
	such that Lie brackets of sections of $\mathcal D$ span a rank three subbundle $[\mathcal D,\mathcal D]\subseteq TM$ 
	and triple brackets of sections of $\mathcal D$ span all of the tangent bundle, $[\mathcal D,[\mathcal D,\mathcal D]]=TM$.
	Clearly, this is a $C^2$-open condition on $\mathcal D$, whence the name \emph{generic rank two distribution.}
	Such a distribution gives rise to a filtration of the tangent bundle by smooth subbundles,
	\begin{equation}\label{E:TM.filt}
		TM\supseteq[\mathcal D,\mathcal D]\supseteq\mathcal D.
	\end{equation}
	The ranks of these filtration bundles are 2, 3, and 5, respectively, whence the name \emph{(2,3,5) distribution.}
	We will denote the associated graded bundle by 
	\begin{equation}\label{E:tM}
		\mathfrak tM=\mathfrak t_{-3}M\oplus\mathfrak t_{-2}M\oplus\mathfrak t_{-1}M
	\end{equation}
	where $\mathfrak t_{-1}M=\mathcal D$, $\mathfrak t_{-2}M=\frac{[\mathcal D,\mathcal D]}{\mathcal D}$, and $\mathfrak t_{-3}M=\frac{TM}{[\mathcal D,\mathcal D]}$.
	The Lie bracket of vector fields induces the structure of a graded nilpotent Lie algebra on the fibers of $\mathfrak tM$.
	Each fiber is isomorphic to the graded nilpotent Lie algebra
	\begin{equation}\label{E:goe.intro}
		\goe=\goe_{-3}\oplus\goe_{-2}\oplus\goe_{-1}	
	\end{equation}
	with graded basis $X_1,X_2\in\goe_{-1}$, $X_3\in\goe_{-2}$, $X_4,X_5\in\goe_{-3}$ and nontrivial brackets
	\begin{equation}\label{E:brackets.intro}
		[X_1,X_2]=X_3,\quad[X_1,X_3]=X_4,\quad[X_2,X_3]=X_5.
	\end{equation}
	This turns $\mathfrak tM$ into a locally trivial bundle of Lie algebras called the bundle of osculating algebras.
	The fiberwise Lie bracket provides canonical isomorphisms
	\begin{equation}\label{E:tM123}
		\Lambda^2\mathfrak t_{-1}M=\mathfrak t_{-2}M
		\quad\text{and}\quad
		\mathfrak t_{-1}M\otimes\mathfrak t_{-2}M=\mathfrak t_{-3}M.
	\end{equation}
	
	We let $G$ denote the simply connected Lie group with Lie algebra $\goe$.
	The left invariant distribution $\mathcal D_G$ spanned by the vector fields $X_1$ and $X_2$ provides a basic example of a (2,3,5) distribution on $G$.
	Up to an (in general ungraded) automorphism every left invariant (2,3,5) distribution on $G$ is of this form, cf.~\cite[Lemma~4.1]{H22}.

	The study of generic rank two distributions in dimension five reaches back to the work of \'Elie Cartan.
	In his seminal paper \cite{C10} Cartan constructed what is nowadays known as canonical Cartan connection.
	This leads to an equivalent description of a (2,3,5) distribution 
	as a parabolic geometry \cite{CS09} associated with the split real form of the exceptional Lie group $G_2$ and a particular parabolic subgroup, see also \cite[\S4.3.2]{CS09} and \cite{S08}.
	In particular, (2,3,5) distributions do have local geometry --- unlike contact and Engel structures \cite{P16}.
	Cartan's harmonic curvature tensor, $\kappa_H\in\Gamma(S^4\mathcal D^*)$, provides a complete obstruction to local flatness, 
	i.e., a $(2,3,5)$ distribution is locally diffeomorphic to the left invariant distribution $\mathcal D_G$ on $G$ if and only if $\kappa_H=0$.
	Generic rank two distributions in dimension five have attracted quite some attention recently, 
	cf.~\cite{AN14,AN18,BHN18,BM09,BH93,CN09,CSa09,DH19,GPW17,H22,H23a,H23b,HS09,HS11,LNS17,N05,S06,S08,SW17,SW17b}.

	The Rumin complex \cite{R99,R01,R05} associated with a (2,3,5) distribution is sequence of natural differential operators 
	\[
		\cdots\to\Gamma\bigl(\mathcal H^q(\mathfrak tM)\bigr)\xrightarrow{D_q}\Gamma\bigl(\mathcal H^{q+1}(\mathfrak tM)\bigr)\to\cdots
	\]
	satisfying $D_{q+1}D_q=0$ and computing de~Rham cohomology. 
	Here the vector bundles $\mathcal H^q(\mathfrak tM)$ are obtained from the bundle of osculating algebras $\mathfrak tM$
	by passing to fiberwise Lie algebra cohomology with trivial real coefficients.
	There exist (non-canonical, in this formulation) injective differential operators $L_q\colon\Gamma(\mathcal H^q(\mathfrak tM))\to\Omega^q(M)$, 
	intertwining the Rumin differential with the de~Rham differential and inducing canonical isomorphisms in cohomology.
	Rumin has shown that this complex becomes exact on the level of the Heisenberg principal symbol.
	Hence, the Rumin complex is a Rockland \cite{R78} complex \cite[\S2.3]{DH22}, the analogue of an elliptic complex in the Heisenberg calculus \cite{M82,EY19}.

	A fiberwise graded Euclidean inner product $g$ on $\mathfrak tM$ and an orientation $\mathfrak o$ of $M$ give rise to a Hodge star operator 
	\[
		\star_q=\star_{\mathcal D,g,\mathfrak o,q}^M\colon\mathcal H^q(\mathfrak tM)\to\mathcal H^{5-q}(\mathfrak tM)
	\]
	see \cite[\S2]{R99}, \cite[Prop~2.8]{R01}, or \cite[\S3.3]{H22}.
	Let $F$ be a unitary flat vector bundle over $M$, i.e., a flat complex vector bundle that admits a parallel fiberwise Hermitian metric $h$.
	Twisting $\mathbf i\star_3 D_2$ with $F$ we obtain an operator 
	\begin{equation}\label{E:S.def.intro}
		S=S_{\mathcal D,g,\mathfrak o}^{M,F}=\bigl(\mathbf i\star_3D_2\bigr)^F
	\end{equation}
	acting on sections of $\mathcal H^2(\mathfrak tM)\otimes F$.
	This operator is formally self-adjoint with respect to the standard $L^2$ inner product provided by $g$, $h$, and the volume density on $M$ associated with $g$ via the canonical isomorphism of lines bundles
	\begin{equation}\label{E:tTM5}
		\Lambda^5TM=\Lambda^5\mathfrak tM.
	\end{equation}
	Note that $S$ does not depend on the choice of $h$.
	
	\begin{theorem}\label{T:eta.S.hinv}
		Let $\mathcal D$ be a generic rank two distribution on a closed 5-manifold $M$,
		let $g$ be a fiberwise graded Euclidean inner product on $\mathfrak tM$,
		let $\mathfrak o$ be an orientation of $M$,
		and let $F$ be a unitary flat vector bundle over $M$.
		Then the associated operator $S=S_{\mathcal D,g,\mathfrak o}^{M,F}$ in~\eqref{E:S.def.intro} is essentially self-adjoint with infinite dimensional kernel.
		The non-zero part of its spectrum, $\spec_*(S)$, consists of isolated real eigenvalues with finite multiplicities only.
		The eta function 
		\begin{equation}\label{E:etaS.intro}
			\eta_S(s):=\sum_{\lambda\in\spec_*(S)}\sign(\lambda)|\lambda|^{-s}
		\end{equation}
		converges absolutely for $\Re s>5$.
		This function extends to an entire function on the complex plane that vanishes at negative odd integers, i.e.,
		\begin{equation}\label{E:eta.S.odd}
			\eta_S(-2l-1)=0,\qquad l=0,1,2,\dotsc
		\end{equation}
		Moreover, $\eta_S(0)$ is independent of $g$ and constant on the connected component of $\mathcal D$ in the space of generic rank two distributions on $M$.
	\end{theorem}

	In view of Theorem~\ref{T:eta.S.hinv}, the eta invariant 
	\begin{equation}\label{E:eta.S.def.intro}
		\eta(S):=\eta_S(0)
	\end{equation}
	only depends on the diffeomorphism type of $M$, the orientation $\mathfrak o$ of $M$, the unitary flat vector bundle $F$ over $M$,
	and the connected component of $\mathcal D$ in the space of generic rank two distributions on $M$.
	We denote this invariant by
	\begin{equation}\label{E:eta.MFDo.intro}
		\eta_F(M,\mathcal D,\mathfrak o):=\eta(S).
	\end{equation}

	Formally putting $s=0$ is \eqref{E:etaS.intro} we see that $\eta(S)$ is, in a zeta regularized sense, the difference between the number of positive and negative eigenvalues of $S$.
	Hence, $\eta(S)$ may be regarded as a measure for the spectral asymmetry of $S$.

	To state our next result let $\Gamma$ be a (cocompact) lattice in the simply connected nilpotent Lie group $G$ with Lie algebra $\goe$, cf.~\eqref{E:goe.intro}.
	This Lie group admits many different lattices which can all be described explicitly, cf.~\cite[Lemma~2]{H23b}.
	The left invariant (2,3,5) distribution $\mathcal D_G$ on $G$ descends to a (2,3,5) distribution on the nilmanifold $M=\Gamma{\setminus}G$ that will be denoted by $\mathcal D_\Gamma$.
	Moreover, let $\mathfrak o$ be an orientation of $\goe$.
	The corresponding orientation $\mathfrak o_G$ of $G$ descends to an orientation on $\Gamma{\setminus}G$ that will be denoted by $\mathfrak o_\Gamma$.

	\begin{theorem}\label{T:eta.nilmf.entire}
		For any lattice $\Gamma\subseteq G$, either orientation $\mathfrak o$ of $\goe$,
		and every unitary flat vector bundle $F$ over the nilmanifold $M=\Gamma{\setminus}G$ we have
		\begin{equation}\label{E:eta.nilmf.S}
			\eta_F(\Gamma{\setminus}G,\mathcal D_\Gamma,\mathfrak o_\Gamma)=0.
		\end{equation}
	\end{theorem}
	
	The results in \cite{APS75,APS75b} imply that the Riemannian analogue, i.e., the eta invariant of the twisted odd signature operator on a (2,3,5) nilmanifold $\Gamma{\setminus}G$ vanishes too.
	This holds for every Riemannian metric on $\Gamma{\setminus}G$ after twisting with any unitary flat vector bundle $F$, see Proposition~\ref{P:APS235} below for details.
	
	For (2,3,5) nilmanifolds $M=\Gamma{\setminus}G$ and fiberwise graded Euclidean inner products of special form the whole eta function $\eta_S(s)$ can be worked out explicitly.
	To this end, let $g$ be a graded Euclidean inner product on $\goe$.
	The corresponding left invariant fiberwise graded Euclidean inner product $g_G$ on $\mathfrak tG$
	descends to a fiberwise graded Euclidean inner product on $\mathfrak t(\Gamma{\setminus}G)$ which we will denote by $g_\Gamma$.
	The graded Euclidean inner product $g$ on $\goe$ determines a number $a_g>0$ and a Euclidean inner product $b_g$ on $\goe_{-1}$ such that
	\begin{align}\label{E:ag}
		g([X,Y],[X,Y])
		&=a_g\cdot\bigl(g(X,X)g(Y,Y)-g(X,Y)^2\bigr)
		\\\label{E:bg}
		g\bigl([X,Z],[Y,Z]\bigr)
		&=b_g(X,Y)\cdot g(Z,Z)
	\end{align}
	for $X,Y\in\goe_{-1}$ and $Z\in\goe_{-2}$.
	In Theorem~\ref{T:eta.nilmf.values} below we identify the whole eta function of the operator $S=S^{\Gamma{\setminus}G,F}_{\mathcal D_\Gamma,g_\Gamma,\mathfrak o_\Gamma}$ explicitly, 
	provided $b_g$ is proportional to $g|_{\goe_{-1}}$.
	To formulate this result we need to introduce more notation.

	For $a\in\mathbb R\setminus\mathbb Z$ we put 
	\begin{equation}\label{E:eta.Hurw}
		\eta_\Hurw(s,a)
		:=\sum_{n=-\infty}^\infty\sign(n+a)|n+a|^{-s}.
	\end{equation}
	Clearly, this is odd and periodic with period one in $a$. 
	If $0<a<1$ then
	\begin{equation}\label{E:eta.zeta.Hurw}
		\eta_\Hurw(s,a)
		=\zeta_\Hurw(s,a)-\zeta_\Hurw(s,1-a)
	\end{equation}
	where
	\begin{equation}\label{E:Hurw}
		\zeta_\Hurw(s,a)=\sum_{n=0}^\infty(n+a)^{-s}
	\end{equation}
	denotes the Hurwitz zeta function \cite[\href{https://dlmf.nist.gov/25.11.i}{\S 25.11(i)}]{NIST:DLMF}.
	Note that $\eta_\Hurw(s,a)$ is an entire function for $\zeta_\Hurw(s,a)$ has a single simple pole at $s=1$ with residue 1.

	Moreover, for $a\notin\{\pm\lambda_n:n=0,1,2,\dots\}$ we put
	\begin{equation}\label{E:tetaa.def}
		\tilde\eta(s,a)
		:=\sum_{n=0}^\infty\sign(a+\lambda_n)|a+\lambda_n|^{-s}
		+\sum_{n=0}^\infty\sign(a-\lambda_n)|a-\lambda_n|^{-s}
	\end{equation}
	with
	\begin{equation}\label{E:lambdan}
		\lambda_n
		:=\frac{\sqrt{8(2n+1)^2+9}}4,
		\qquad n=0,1,2,\dotsc
	\end{equation}
	These values $\lambda_n$ are closely related to the spectrum of the operator $S$ in Schr\"o\-din\-ger representations of $G$, cf.~\eqref{E:spec.S.h} below.
		
	\begin{lemma}\label{L:tetaa}
		The series in \eqref{E:tetaa.def} converge absolutely for $\Re s>1$ and $\tilde\eta(s,a)$ extends to a meromorphic function on the entire complex plane
		that has only simple poles located at $s=-2l$ with $l=1,2,3,\dotsc$ and residues
		\begin{equation}\label{E:tetaa.res}
			\res_{s=-2l}\tilde\eta(s,a)
			=\sqrt2\cdot\sum_{j=0}^{l-1}\binom{2l}{2j+1}\binom{2l-2j}{l-j}\left(\frac9{8\cdot8}\right)^{l-j}a^{2j+1}.
		\end{equation}
		Moreover,
		\begin{equation}\label{E:tetaa.zero}
			\tilde\eta(0,a)=2\sign(a)\cdot\#\bigl\{n\in\mathbb N_0:\lambda_n<|a|\bigr\}-\sqrt2\cdot a
		\end{equation}
		and 
		\begin{equation}\label{E:tetaa.odd}
			\tilde\eta(-2l-1,a)=0,\qquad l=0,1,2,\dots
		\end{equation}
	\end{lemma}

	Every lattice $\Gamma$ in $G$ gives rise to a canonical short exact sequence of abelian groups \cite[\S2]{H23b}
	\begin{equation}\label{E:ses}
		0\to\frac{\Gamma\cap[G,G]}{[\Gamma,\Gamma]\cdot(\Gamma\cap C)}\to\frac\Gamma{[\Gamma,\Gamma]\cdot(\Gamma\cap C)}\to\frac\Gamma{\Gamma\cap[G,G]}\to0
	\end{equation}
	where $C$ denotes the center of $G$, $\frac\Gamma{\Gamma\cap[G,G]}\cong\mathbb Z^2$, and $\frac{\Gamma\cap[G,G]}{[\Gamma,\Gamma]\cdot(\Gamma\cap C)}\cong\mathbb Z/r\mathbb Z$ 
	is a finite cyclic group whose order shall be denoted by $r=r_\Gamma$.
	There are no restrictions on this order; for each $r\in\mathbb N$ there exists a lattice giving rise to a group of order $r$, see \cite[Lemma~1]{H23b}.
	As the sequence in \eqref{E:ses} splits we obtain the following description of unitary characters $\chi\colon\Gamma\to U(1)$ that vanish on $\Gamma\cap C$:
	\[
		\hom\bigl(\tfrac\Gamma{\Gamma\cap C},U(1)\bigr)
		\cong U(1)\times U(1)\times\hom\bigl(\mathbb Z/r\mathbb Z,U(1)\bigr).
	\]
	In particular, every unitary character of $\frac{\Gamma\cap[G,G]}{[\Gamma,\Gamma]\cdot(\Gamma\cap C)}\cong\mathbb Z/r\mathbb Z$ 
	can be extended to a unitary character of $\Gamma$ that vanishes on $\Gamma\cap C$.

	\begin{theorem}\label{T:eta.nilmf.values}
		In the situation of Theorem~\ref{T:eta.nilmf.entire} assume, moreover, 
		that $F=F_\chi$ is the unitary flat line bundle over the nilmanifold $M=\Gamma{\setminus}G$ associated with a unitary character $\chi\colon\Gamma\to U(1)$.
		Furthermore, let $g_\Gamma$ denote the fiberwise graded Euclidean inner product on $\mathfrak t(\Gamma{\setminus}G)$ 
		associated with a graded Euclidean inner product $g$ on $\goe$ for which $b_g$ is proportional to $g|_{\goe_{-1}}$, cf.~\eqref{E:bg}.
		Then the eta function of the associated operator $S=S^{\Gamma{\setminus}G,F}_{\mathcal D_\Gamma,g_\Gamma,\mathfrak o_\Gamma}$ in~\eqref{E:S.def.intro} has the following properties:

		(a) If $\chi|_{\Gamma\cap C}$ is nontrivial, then $\eta_S(s)=0$ for all $s$.

		(b) If $\chi|_{\Gamma\cap[G,G]}$ is trivial, then $\eta_S(s)=0$ for all $s$.

		(c) If $\chi|_{\Gamma\cap C}$ is trivial and $\chi|_{\Gamma\cap[G,G]}$ is nontrivial, then
		\begin{equation}\label{E:eta.S.formula.intro}
			\eta_S(s)
			=r\cdot\left(\frac{2\pi}{\sqrt{g(\gamma,\gamma)}}\right)^{-s}
			\cdot\eta_\Hurw\bigl(s-1,\tfrac cr\bigr)
			\cdot\tilde\eta\bigl(s,\tfrac54\bigr)
		\end{equation}
		where $r=r_\Gamma$ denotes the order of $\frac{\Gamma\cap[G,G]}{[\Gamma,\Gamma]\cdot(\Gamma\cap C)}\cong\mathbb Z/r\mathbb Z$, 
		$\gamma=\gamma_{\Gamma,\mathfrak o}\in\goe_{-2}=\frac{[G,G]}C\cong\mathbb R$ is the positive (with respect to the orientation of $\goe_{-2}$ induced by $\mathfrak o$) 
		generator of $\frac{\Gamma\cap[G,G]}{\Gamma\cap C}\cong\mathbb Z$, 
		and $c=c_{\Gamma,\mathfrak o,\chi}$ denotes an integer (unique and non-zero mod $r$) such that $e^{2\pi\mathbf ic/r}=\chi(\gamma)$.
		Moreover, $\eta_S(0)=0$ and for $l=1,2,3,\dotsc$ we have
		\begin{multline}\label{E:eta.S.2l}
			\eta_S(-2l)
			=\frac{(-1)^l\sqrt2\cdot r}{2\pi\cdot g(\gamma,\gamma)^l}
			\cdot(2l+1)!\cdot\Im\Li_{2l+2}\bigl(e^{2\pi\mathbf ic/r}\bigr)
			\\
			\cdot\sum_{j=0}^{l-1}\binom{2l}{2j+1}\binom{2l-2j}{l-j}\left(\frac9{8\cdot8}\right)^{l-j}\left(\frac54\right)^{2j+1}
		\end{multline}
		where $\Li_s(z)=\sum_{n=1}^\infty\frac{z^n}{n^s}$ denotes the polylogarithm \cite[\href{https://dlmf.nist.gov/25.12.ii}{\S25.12(ii)}]{NIST:DLMF}. 
	\end{theorem}
	
	Theorem~\ref{T:eta.nilmf.values} clearly shows that on a (2,3,5) nilmanifold the eta function of $S$ is nontrivial in general and, hence, 
	the spectrum of $S$ need not be symmetrical about zero.
	More specifically, we have:

	\begin{corollary}
		In the situation of Theorem~\ref{T:eta.nilmf.values}(c) and for $l=1,2,\dotsc$ we have 
		\begin{align*}
			(-1)^l\eta_S(-2l)>0&\qquad\text{if $\tfrac cr\in(0,\tfrac12)+\mathbb Z$ and}\\
			(-1)^l\eta_S(-2l)<0&\qquad\text{if $\tfrac cr\in(\tfrac12,1)+\mathbb Z$.}
		\end{align*}
		If $\frac cr\in\frac12+\mathbb Z$, then $\eta_S(s)$ vanishes identically.
	\end{corollary}

	\begin{proof}
		If $\frac cr\in\frac12+\mathbb Z$, then $\eta_\Hurw(s,\frac cr)=0$ in view of $\eta_\Hurw(s,a+1)=\eta_\Hurw(s,a)$ and $\eta_\Hurw(s,-a)=-\eta_\Hurw(s,a)$.
		Hence, in this case $\eta_S(s)$ vanishes identically by \eqref{E:eta.S.formula.intro}.
		The integral formula for the polylogarithm in \cite[\href{https://dlmf.nist.gov/25.12.E11}{Eq.~(25.12.11)}]{NIST:DLMF},
		\[
			\Li_s(z)=\frac z{\Gamma(s)}\int_0^\infty\frac{x^{s-1}dx}{e^x-z},\qquad\Re s>0,\, z\in\mathbb C\setminus[1,\infty)
		\]
		yields
		\begin{equation}\label{E:Im.Li.even}
			(2l+1)!\cdot\Im\Li_{2l+2}(e^{2\pi\mathbf ia})
			=\sin(2\pi a)\int_0^\infty\frac{x^{2l+1}e^xdx}{|e^x-e^{2\pi\mathbf i a}|^2}
		\end{equation}
		for $a\in\mathbb R\setminus\mathbb Z$ and $l=0,1,2,\dotsc$
		Hence, $(2l+1)!\cdot\Im\Li_{2l+2}(e^{2\pi\mathbf ia})$ has the same sign as $\sin(2\pi a)$ for $a\notin\frac12\mathbb Z$.
		The remaining assertions thus follow from \eqref{E:eta.S.2l}.
	\end{proof}

	Let us close this introduction with a few preliminary remarks concerning the proofs of the aforementioned theorems.

	The operator $S$ in \eqref{E:S.def.intro} is closely related to the Rockland operator 
	\[
		\Delta=\Delta^{M,F}_{\mathcal D,g}=\bigl((D_1D_1^*)^2+(D_2^*D_2)^3\bigr)^F.
	\]
	The short time asymptotic expansion of $\tr(Se^{-t\Delta})$ provides the analytic continuation of $\eta_S(s)$ via Mellin transform,
	\begin{equation}\label{E:eta.cont.intro}
		\eta_S(s)
		=\tr\left(S\Delta^{-(s+1)/6}\right)
		=\frac1{\Gamma\left(\frac{s+1}6\right)}\int_0^\infty t^{(s+1)/6}\tr\bigl(Se^{-t\Delta}\bigr)\frac{dt}t,\quad\Re s>5.
	\end{equation}
	The coefficients in this expansion give all residues of $\eta_S(s)$ and the values at $s=-1,-7,-13,\dotsc$
	For the values at the remaining negative odd integers we replace $S$ with odd powers of $S$, cf.~Proposition~\ref{P:etaSs}.
	As the coefficients in the asymptotic expansion of $\tr(S^{1+2r}e^{-t\Delta})$ are locally computable, 
	they coincide with the coefficients associated with the trivial flat vector bundle, $F_\triv=M\times\mathbb C^{\rk F}$.
	However, the eta function associated with $F_\triv$ vanishes identically since the spectrum of $S$ is symmetrical about zero in this case, cf.~\cite[Remark~3a, p~61]{APS75}.
	Hence, the coefficients in the aforementioned expansion vanish, for all unitary flat bundles $F$, see Proposition~\ref{P:loc.quant} below.
	This yields everything but the last part of Theorem~\ref{T:eta.S.hinv}.
	For the remaining statement we show that the variation of the eta invariant with respect to $g$ and $\mathcal D$ is locally computable
	and argue as before to see that this variation must vanish too.
	
	The proof of Theorem~\ref{T:eta.nilmf.values} is based on the decomposition 
	\begin{equation}\label{E:S.deco}
		S=\bigoplus_\rho m(\rho)\cdot\rho\bigl(S_G\bigr),
	\end{equation}
	where $\rho$ runs through the irreducible unitary representations of $G$, and 
	\[
		S_G:=S^{G.\mathbb C}_{\mathcal D_G,g_G,\mathfrak o_G}=\mathbf i\star_3D_2\in\mathcal U(\goe)\otimes\eend\bigl(H^2(\goe)\bigr)
	\]
	denotes the left invariant differential operator on $G$ associated with the left invariant (2,3,5) distribution $\mathcal D_G$.
	The multiplicities $m(\rho)$ are known explicitly through a formula of Howe \cite{H71} and Richardson \cite{R71}.
	The eta function decomposes accordingly,
	\begin{equation}\label{E:S.eta.deco}
		\eta_S(s)=\sum_\rho m(\rho)\cdot\eta_{\rho(S_G)}(s).
	\end{equation}
	All irreducible unitary representations of $G$ have been described explicitly by Dixmier \cite[Prop~8]{D58}, see also \cite{K62,K04}.
	There are three types of these representations: scalar representations, Schr\"odinger representations, and generic representations.
	For scalar representations $\rho$, the spectrum of $\rho(S_G)$ is symmetrical about zero, hence $\eta_{\rho(S_G)}(s)=0$, see Lemma~\ref{L:eta.scalar} below.
	The same is true for generic representations $\rho$, see Lemma~\ref{L:eta.generic}(b) below.
	For Schr\"odinger representations $\rho$, the spectrum of $\rho(S_G)$ has been determined in \cite[Lemma~3]{H23a} and this permits to compute
	residues and special values of $\eta_{\rho(S_G)}$, see Lemma~\ref{L:eta.Schroedinger}(b) below.
	Combining this with \eqref{E:S.eta.deco} and the known multiplicities $m(\rho)$ we obtain the explicit formula \eqref{E:eta.S.formula.intro} stated in Theorem~\ref{T:eta.nilmf.values}(c).
	In case (a) of Theorem~\ref{T:eta.nilmf.values} no Schr\"odinger representations appear in the decomposition \eqref{E:S.deco} and $\eta_S$ vanishes identically.
	In case (b) the contributions from the Schr\"odinger representations cancel in pairs.

	Theorem~\ref{T:eta.nilmf.entire} can be derived from Theorems~\ref{T:eta.S.hinv} and \ref{T:eta.nilmf.values} using the fact that 
	every irreducible unitary representation $\rho\colon\Gamma\to U(N)$ is induced from a unitary character on a sublattice \cite{B73}.

	There are well-understood topological obstructions to the existence of (2,3,5) distributions with prescribed Euler class, cf.~\cite[Thm~1]{DH19}.
	One motivation for our study of the Rumin complex is the natural question: Are there further, geometrical obstructions to the existence of (2,3,5) distributions on closed 5-manifolds?
	More specifically: To what extent do (2,3,5) distributions abide by an h-principle \cite{EM02,G86}?
	Similar questions have led to interesting discoveries in contact and Engel topology \cite{BEM15,CPdPP17,CdPP20,E89,G91,dPV20,V09}.
	The analytic torsion of the Rumin complex associated with a (2,3,5) distribution has been discussed in \cite{H22}.
	For nilmanifolds this analytic torsion coincides with the Ray--Singer torsion, see \cite{H23b} and \cite[Thm~1.3]{H22}.

	The remaining part of this paper is organized as follows.
	In Section~\ref{S:eta} we discuss basic properties of the eta function of self-adjoint Rockland differential operators on closed filtered manifolds.
	In Section~\ref{S:eta235} we specialize to the operator $S$ associated with a (2,3,5) distribution and give a proof of Theorem~\ref{T:eta.S.hinv}.
	In Section~\ref{S:reps} we discuss the eta function of $S_G$ in irreducible unitary representations of $G$.
	In Section~\ref{S:nilmf} we combine these results to prove Theorems~\ref{T:eta.nilmf.entire} and \ref{T:eta.nilmf.values}.

\section{The eta invariant on filtered manifolds}\label{S:eta}

	The analysis available in the literature permits to give a rigorous definition of the eta invariant \cite{APS73,APS75} 
	of (formally) self-adjoint Rockland differential operators on general closed filtered manifolds
	by generalizing the approach used in \cite[\S9]{BHR07} and \cite[Chapter~1.10]{G84}.
	In this section we will recall these arguments and point to the required subelliptic analysis \cite{M82,EY19,DH20,DH22}.

	To this end let $A=A^*$ be a formally self-adjoint differential operator of Heisenberg order $\kappa>0$ acting on 
	sections of a complex vector bundle $E$ over a closed filtered manifold $M$ of homogeneous dimension $n$.
	Here the adjoint is with respect to a standard $L^2$ inner product on sections of $E$ constructed using a fiberwise Hermitian inner product on $E$ and a volume density on $M$.
	Assume that $A$ satisfies the (pointwise) Rockland condition \cite[Def~2.5]{DH22}.
	Then $A$ admits a parametrix in the Heisenberg calculus \cite[Thm~A]{DH22}.
	Consequently, $A$ is essentially self-adjoint with compact resolvent \cite[Lemma~1]{DH20} and, hence, has a spectrum that consists of isolated eigenvalues with finite multiplicities.

	Clearly, $A^2=A^*A=|A|^2$ is a Rockland differential operator of Heisenberg order $2\kappa>0$.
	Moreover, $A|A|^{-s-1}=A(A^2)^{-(s+1)/2}$ is a pseudodifferential operator of Heisenberg order $-\kappa s$ in the Heisenberg calculus \cite[Thm~2]{DH20}, 
	whence trace class for $\Re s>n/\kappa$, see \cite[Prop~3.7(d)]{DH22}.
	Throughout this paper the complex powers of an operator are understood to vanish on its kernel.
	Put
	\begin{equation}\label{E:eta.def}
		\eta_A(s):=\tr\left(A|A|^{-s-1}\right)=\sum_{\lambda\in\spec_*(A)}\sign(\lambda)|\lambda|^{-s},\qquad\Re s>n/\kappa
	\end{equation}
	where $\spec_*(A)$ denotes the nonzero spectrum of $A$, i.e., the sum is over all nonzero eigenvalues $\lambda$ of $A$, repeated according to their multiplicity.

	The Schwartz kernel of $Ae^{-tA^2}$ admits an asymptotic expansion along the diagonal of the form
	\[
		\bigl(Ae^{-tA^2}\bigr)(x,x)
		\sim\sum_{j=0}^\infty t^{(j-\kappa-n)/2\kappa}\mathbf a_j(A,A^2)(x),\qquad\text{as $t\to0$,}
	\]
	where $\mathbf a_j(A,A^2)\in\Gamma^\infty\bigl(\eend(E)\otimes|\Lambda|_M\bigr)$ is locally computable, and $\mathbf a_j(A,A^2)=0$ if $j-\kappa$ is odd, cf.~\cite[Lemma~2.18]{H22}.
	Here $|\Lambda|_M$ denotes the bundle of densities on $M$.
	Taking the pointwise trace and integrating over $M$ we obtain the expansion 
	\begin{equation}\label{E:trAe-tA2}
		\tr\bigl(Ae^{-tA^2}\bigr)
		\sim\sum_{j=0}^\infty t^{(j-\kappa-n)/2\kappa}a_j(A,A^2),\qquad\text{as $t\to0$,}
	\end{equation}
	with coefficients $a_j(A,A^2)=\int_M\tr_E(\mathbf a_j(A,A^2))$ and $a_j(A,A^2)=0$ if $j-\kappa$ is odd.
	
	\begin{proposition}\label{P:eta.cont}
		The function $\eta_A(s)$ extends meromorphically to the entire complex plane.
		Its poles are all simple and they can only be located at $s=(n-j)/\kappa$ with $j\in\mathbb N_0$ such that $j-\kappa$ is even.
		The residues are
		\begin{equation}\label{E:etaA.res}
			\res_{s=\frac{n-j}\kappa}\eta_A(s)=2\cdot\frac{a_j(A,A^2)}{\Gamma\left(\frac{n+\kappa-j}{2\kappa}\right)},
		\end{equation}
		where $a_j(A,A^2)$ are the coefficients in the expansion \eqref{E:trAe-tA2}.
		In particular, 
		\begin{equation}\label{E:eta.A.res0}
			\res_{s=0}\eta_A(s)=2\cdot\frac{a_n(A,A^2)}{\sqrt\pi}
		\end{equation}
		and the eta function is regular at $s=0$ if $n-\kappa$ is odd.
		If $l\in\mathbb N_0$, then the eta function is regular at $s=-2l-1$,
		\begin{equation}\label{E:etaA.value}
			\eta_A(-2l-1)=(-1)^ll!\cdot a_{n+(2l+1)\kappa}(A,A^2),
		\end{equation}
		and this vanishes for odd $n$.
	\end{proposition}

	\begin{proof}
		Via Mellin transform we get from \eqref{E:eta.def}
		\begin{equation}\label{E:Mellin}
			\eta_A(s)=\frac1{\Gamma\left(\frac{s+1}2\right)}\int_0^\infty t^{(s+1)/2}\tr\left(Ae^{-tA^2}\right)\frac{dt}t,\qquad\Re s>n/\kappa.
		\end{equation}
		The integral converges at infinity since there exists $c>0$ such that $\tr(Ae^{-tA^2})=O(e^{-ct})$ as $\to\infty$.
		For $\Re s>n/\kappa$ the convergence at zero follows from the asymptotic expansion in \eqref{E:trAe-tA2}.
		Moreover,
		\begin{equation}\label{E:eta.cont}
			\eta_A(s)=\frac1{\Gamma\left(\frac{s+1}2\right)}\left(\sum_{j=0}^{J-1}\frac{2\cdot a_j(A,A^2)}{s+(j-n)/\kappa}+h_J(s)\right)
		\end{equation}
		where $h_J(s)$ is holomorphic for $\Re s>(n-J)/\kappa$.
		This provides the analytic continuation of $\eta_A(s)$ with simple poles and residues as in \eqref{E:etaA.res}.
		The special values in \eqref{E:etaA.value} can be read off using 
		\begin{equation}\label{E:Gamma.res}
			\frac1{\Gamma(s)}=(-1)^ll!\cdot(s+l)+O\bigl((s+l)^2\bigr),\qquad\text{as $s\to-l$.}
		\end{equation}
		Using $\Gamma(\frac12)=\sqrt\pi$ we obtain \eqref{E:eta.A.res0} from \eqref{E:etaA.res} by putting $j=n$.
	\end{proof}

	The eta function might not be regular at $s=0$ in this generality, contrary to the Riemannian case, cf.~\cite{APS76} or \cite[\S1.10]{G84}.
	We define the eta invariant of $A$ as the constant part in the Laurent expansion at the origin, 
	\begin{equation}\label{E:etaA.def}
		\eta(A):=\FP_{s=0}\eta_A(s).
	\end{equation}

	We have a variation formula analoguous to the one in \cite[Prop~2.10]{APS76}, \cite[\S9]{BHR07} or \cite[Lemma~1.10.2]{G84}.

	\begin{proposition}\label{P:dot.A}
		Suppose $A_u$ is a smooth family of (formally) self-adjoint Rockland differential operators of Heisenberg order $\kappa>0$ 
		such that $\ker(A_u)$ has constant dimension where $u\in\mathbb R$.
		Then
		\begin{equation}\label{E:eta.dot}
			\tfrac\partial{\partial u}\eta_{A_u}(s)=-s\cdot\tr\left(\dot A_u|A_u|^{-s-1}\right)
		\end{equation}
		where $\dot A_u:=\frac\partial{\partial u}A_u$.
		In particular, $\res_{s=0}\eta_{A_u}(s)$ is constant in $u$ and 
		\[
			\tfrac\partial{\partial u}\eta(A_u)=-2\cdot\frac{a_n(\dot A_u,A_u^2)}{\sqrt\pi}
		\]
		where $a_j(\dot A_u,A_u^2)$ denote the constants in the asymptotic expansion 
		\begin{equation}\label{E:Adot.asymp}
			\tr\bigl(\dot A_ue^{-tA^2_u}\bigr)
			\sim\sum_{j=0}^\infty t^{(j-\kappa-n)/2\kappa}a_j(\dot A_u,A_u^2),\qquad\text{as $t\to0$,}
		\end{equation}
		and $a_j(\dot A_u,A_u^2)=0$ if $j-\kappa$ is odd.
	\end{proposition}

	\begin{proof}
		The asymptotic expansion in \eqref{E:Adot.asymp} follows from \cite[Lemma~2.18]{H22}.
		Let $\Pi_u$ denote the orthogonal projection onto $\ker(A_u)$.
		Note that $\Pi_u$ depends smoothly on $u$ since we assumed that $\ker(A_u)$ has constant dimension.
		Differentiating $A_u\Pi_u=0$ we obtain $\dot A_u\Pi_u+A_u\dot\Pi_u=0$ and then $\Pi_u\dot A_u\Pi_u=0$, for we also have $\Pi_uA_u=0$ by symmetry of $A_u$.
		Hence, there exists $c>0$ such that $\tr\bigl(\dot A_ue^{-tA^2_u}\bigr)=O(e^{-ct})$ as $t\to\infty$.

		In view of the latter estimate Duhamel's formula applied to \eqref{E:Mellin} yields
		\begin{align*}
			\tfrac\partial{\partial u}\eta_{A_u}(s)
			&=\frac1{\Gamma\left(\frac{s+1}2\right)}\int_0^\infty t^{(s+1)/2}\left(\tr\left(\dot A_ue^{-tA^2_u}\right)
			-2t\tr\left(\dot A_uA_u^2e^{-tA^2_u}\right)\right)\frac{dt}t
			\\&=\frac1{\Gamma\left(\frac{s+1}2\right)}\int_0^\infty t^{(s+1)/2}\left(\tr\left(\dot A_ue^{-tA^2_u}\right)
			+2t\frac\partial{\partial t}\tr\left(\dot A_ue^{-tA^2_u}\right)\right)\frac{dt}t
		\end{align*}
		for $\Re s>n/\kappa$, and integration by parts leads to, cf.~\cite[Lemma~1.10.2]{G84},
		\[
			\tfrac\partial{\partial u}\eta_{A_u}(s)
			=\frac{-s}{\Gamma\left(\frac{s+1}2\right)}\int_0^\infty t^{(s+1)/2}\tr\left(\dot A_ue^{-tA^2_u}\right)\frac{dt}t
			=-s\cdot\tr\left(\dot A_u|A_u|^{-s-1}\right),
		\]
		whence \eqref{E:eta.dot} for $\Re s>n/\kappa$.

		Proceeding as in the proof of Proposition~\ref{P:eta.cont} we see that $\tr\bigl(\dot A_u|A_u|^{-s-1}\bigr)$ 
		extends to a meromorphic function with simple poles on the entire complex plane and 
		\begin{equation}\label{E:res.dot.0}
			\res_{s=0}\tr\bigl(\dot A_u|A_u|^{-s-1}\bigr)
			=2\cdot\frac{a_n(\dot A_u,A_u^2)}{\sqrt\pi}.
		\end{equation}
		By analyticity \eqref{E:eta.dot} remains true for all $s$.
		Combining \eqref{E:res.dot.0} with \eqref{E:eta.dot} we obtain
		\[
			\tfrac\partial{\partial u}\eta_{A_u}(s)=-2\cdot\frac{a_n(\dot A_u,A_u^2)}{\sqrt\pi}+O(s),\qquad\text{as $s\to0$.}
		\]
		Whence, the remaining assertions in the propositions, cf.~\eqref{E:etaA.def}.
	\end{proof}

	\begin{remark}\label{R:resA.Wod}
		The residue $\res_{s=0}\eta_A(s)$ is locally computable in light of~\eqref{E:eta.A.res0}.
		Using Proposition~\ref{P:dot.A} one readily shows that this residue only depends on the homotopy class of the Heisenberg principal symbol of $A$.
		Moreover,
		\begin{equation}\label{E:etaA.Wod}
			\kappa\cdot\res_{s=0}\eta_A(s)
			=\tau(\sign(A))
		\end{equation}
		where the right hand side denotes the analogue of Wodzicki's noncommutative residue \cite{W87} in the Heisenberg calculus of the operator 
		\[
			\sign(A)=A|A|^{-1}=A(A^2)^{-1/2}
		\]
		cf.~\cite{P07} and \cite[\S7]{DH20} or \cite{CY24} for the generality required here.
		To see this note that $\sign(A)$ is a pseudodifferential operator of Heisenberg order zero and recall that its noncommutative residue can be expressed in the form \cite[p.~377]{DH20}
		\[
			\tau(\sign(A))
			=2\kappa\cdot\res_{s=0}\tr\bigl(\sign(A)(A^2)^{-s}\bigr).
		\]
		Moreover, from~\eqref{E:etaA.def} we have $\eta_A(s)=\tr\bigl(\sign(A)(A^2)^{-s/2}\bigr)$, whence \eqref{E:etaA.Wod}.
	\end{remark}

\section{The eta invariant of a (2,3,5) distribution}\label{S:eta235}

	Let $\mathcal D\subseteq TM$ be a generic rank two distribution on a closed 5-manifold $M$.
	Recall that such a distribution induces a filtration of the tangent bundle as indicated in \eqref{E:TM.filt}.
	The associated graded bundle $\mathfrak tM$, cf.~\eqref{E:tM}, is a bundle of graded nilpotent Lie algebras when equipped with the fiberwise bracket induced by the Lie bracket of vector fields.
	The osculating algebras, i.e., the fibers of $\mathfrak tM$ are all isomorphic to the graded nilpotent Lie algebra $\goe$ in \eqref{E:goe.intro} with brackets as in \eqref{E:brackets.intro}.

	Let $F$ be a flat complex vector bundle over $M$.
	Twisting the Rumin complex \cite{R99,R01,R05} with $F$ we obtain a sequence of differential operators $D_q=D_{\mathcal D,q}^F$,
	\begin{equation}\label{E:Rumin.complex}
		\Gamma(E_0)\xrightarrow[k_0=1]{D_0}
		\Gamma(E_1)\xrightarrow[k_1=3]{D_1}
		\Gamma(E_2)\xrightarrow[k_2=2]{D_2}
		\Gamma(E_3)\xrightarrow[k_3=3]{D_3}
		\Gamma(E_4)\xrightarrow[k_4=1]{D_4}
		\Gamma(E_5)
	\end{equation}
	where $E_q:=\mathcal H^q(\mathfrak tM)\otimes F$ and $\mathcal H^q(\mathfrak tM)$ denotes the real vector bundle obtained from the bundle of osculating algebras, 
	$\mathfrak tM$, by passing to fiberwise Lie algebra cohomology with trivial real coefficients.
	The Betti numbers are $\rk\mathcal H^q(\mathfrak tM)=\dim H^q(\mathfrak g)=1,2,3,3,2,1$ for $q=0,\dotsc,5$, 
	and the Rumin differential $D_q$ has Heisenberg order $k_q=1,3,2,3,1$ for $q=0,\dotsc,4$, as indicated below the arrows in \eqref{E:Rumin.complex}.
	For more details we refer to \cite[\S5]{BEGN19}, \cite[Example~4.21]{DH22}, and \cite[\S3.1]{H22}.

	The Rumin complex in \eqref{E:Rumin.complex} computes de~Rham cohomology \cite{R99,R01,R05}.
	More precisely, we have $D_{q+1}D_q=0$ and there exist (non-canonical, in this formulation) differential operators $L_q\colon\Gamma(E_q)\to\Omega^q(M;F)$
	intertwining the Rumin differential with the de~Rham differential and inducing canonical isomorphisms \cite[Lemma~3.2]{H22}
	\begin{equation}\label{E:HRumin}
		\frac{\ker D_q}{\img D_{q-1}}=H^q(M;F).
	\end{equation}

	Let $g$ be a graded fiberwise Euclidean metric $g$ on $\mathfrak tM$, i.e., the sum in \eqref{E:tM} is assumed to be orthogonal.
	Let $\partial_q^*\colon\Lambda^{q+1}\mathfrak t^*M\to\Lambda^q\mathfrak t^*M$ denote the fiberwise adjoint of 
	the Chevalley--Eilenberg Lie algebra codifferential $\partial_q\colon\Lambda^q\mathfrak t^*M\to\Lambda^{q+1}\mathfrak t^*M$
	with respect to the induced fiberwise Euclidean metrics on $\Lambda^q\mathfrak t^*M$.
	Fiberwise Hodge decomposition yields an induced fiberwise Euclidean metric on
	\begin{equation}\label{E:HtM}
		\mathcal H^q(\mathfrak tM)
		=\frac{\ker\partial_q}{\img\partial_{q-1}}
		=\ker(\partial_{q-1}^*)\cap\ker(\partial_q).
	\end{equation}
	Moreover, $g$ provides a volume density on $M$ via the canonical isomorphism in~\eqref{E:tTM5}.
	Together with a fiberwise Hermitian metric $h$ on $F$ we obtain fiberwise Hermitian metrics on the bundles $E_q$ and, in turn, 
	standard $L^2$-inner products on $\Gamma(E_q)$ which gives rise to formal adjoints of the Rumin differentials denoted by $D_q^*=D_{\mathcal D,F,g,h,q}^*$.

	Let $\mathfrak o$ be an orientation of $M$.
	Note that the isomorphisms in \eqref{E:tM}, \eqref{E:tM123}, and \eqref{E:tTM5} provide a canonical one-to-one correspondence between  
	orientations of $\mathfrak t_{-1}M$, 
	orientations of $\mathfrak t_{-2}M$, 
	orientations of $\mathfrak t_{-3}M$, 
	orientations of $\mathfrak tM$, and
	orientations of $M$.
	The Hodge star operator associated with $g$ and $\mathfrak o$ preserves the subbundle on the right hand side in \eqref{E:HtM}.
	Twisting with $F$ we obtain fiberwise isometries $\star_q=\star_{\mathcal D,g,\mathfrak o,q}^F\colon E_q\to E_{5-q}$ 
	satisfying $\star^2=1$, since $\dim M=5$ is odd, see \cite[\S2]{R99}, \cite[Prop~2.8]{R01} or \cite[\S3.3]{H22} for details.
	Hence, 
	\[
		\star_q^{-1}=\star_q^*=\star_{5-q}.
	\]
	Assuming, moreover, $\nabla h=0$ we have $D_2^*=-\star_3D_2\,\star_3$.
	Hence, the operator
	\begin{equation}\label{E:S}
		S=S^{M,F}_{\mathcal D,g,\mathfrak o}
		=\mathbf i\star_3^FD_2^F
		=\bigl(\mathbf i\star_3D_2\bigr)^F
	\end{equation}
	acting on $\Gamma(E)$, where we abbreviate $E:=E_2$, satisfies
	\begin{equation}\label{E:S*S}
		S^*=S\qquad\text{as well as}\qquad S^2=D_2^*D_2.
	\end{equation}
	Note that $S$ and $D_q^*$ do not depend on $h$ for we assumed $h$ to be parallel.

	Proposition~\ref{P:eta.cont} is not directly applicable to $S$ since this is not a Rockland operator.
	However, $S$ is closely related to the Rumin--Seshadri \cite{RS12} type operator 
	\begin{equation}\label{E:Delta2.def}
		\Delta:=(D_1D_1^*)^2+(D_2^*D_2)^3
	\end{equation}
	which is a Rockland operator, cf.~\cite[Lemma~2.14]{DH22}.

	\begin{proposition}\label{P:etaSs}
		The operator $S$ has an infinite dimensional kernel but the remaining spectrum consists of isolated eigenvalues of finite multiplicity only.
		Moreover, $|S|^{-s}$ is trace class for $\Re s>5$ and the eta function,
		\begin{equation}\label{E:etaS}
			\eta_S(s):=\tr\left(S|S|^{-s-1}\right)=\sum_{\lambda\in\spec_*(S)}\sign(\lambda)|\lambda|^{-s},\qquad\Re s>5,
		\end{equation}
		extends to a meromorphic function on the entire complex plane with isolated simple poles that can only be located at $s=(10-j)/2$ with $j=0,2,4,\dotsc$
		The residues are
		\begin{equation}\label{E:etaS.res}
			\res_{s=\frac{10-j}2}\eta_S(s)
			=6\cdot\frac{a_j(S^{1+2r},\Delta)}{\Gamma\left(\frac{12+4r-j}{12}\right)},\qquad j,r\in\mathbb N_0
		\end{equation}
		where $a_j(S^{1+2r},\Delta)$ denote the coefficients in the asymptotic expansion
		\begin{equation}\label{E:SDelta.asym}
			\tr\left(S^{1+2r}e^{-t\Delta}\right)
			\sim\sum_{j=0}^\infty t^{(j-12-4r)/12}a_j(S^{1+2r},\Delta),\qquad\text{as $t\to0$,}
		\end{equation}
		and $a_j(S^{1+2r},\Delta)=0$ for odd $j$.
		Furthermore, $\eta_S(s)$ is regular at negative odd integers and
		\begin{equation}\label{E:etaS.value}
			\eta_S(-1-2r-6l)=(-1)^ll!\cdot a_{12+4r+12l}(S^{1+2r},\Delta),\qquad l,r\in\mathbb N_0.
		\end{equation}
	\end{proposition}

	\begin{proof}
		Combining \eqref{E:S*S} and \eqref{E:Delta2.def} we obtain
		\begin{equation}\label{E:Delta2.S6}
			\Delta=(D_1D_1^*)^2+S^6.
		\end{equation}
		Using $D_2D_1=0$ and \eqref{E:S} we get $\img D_1\subseteq\ker S$ and
		\begin{equation}\label{E:D1S}
			(D_1D_1^*)S=0=S(D_1D_1^*).
		\end{equation}
		In particular, $S$ has infinite dimensional kernel.
		As $\Delta$ is a Rockland operator of Heisenberg order $12$ it has compact resolvent and discrete spectrum \cite[Lemma~1]{DH20}.
		Using \eqref{E:Delta2.S6} and \eqref{E:D1S} we conclude that
		\begin{equation}\label{E:spec.ASD}
			\spec_*(\Delta)=\spec_*(D_1D_1^*)^2\sqcup\spec_*(S^6).
		\end{equation}
		Hence, the nonzero spectrum of $S$ is discrete.
		Moreover, $\Delta^{-s}$ is trace class for $\Re s>10/12$, see \cite[Cor~2]{DH22}.
		Combining this with \eqref{E:spec.ASD} we conclude that $|S|^{-s}$ is trace class for $\Re s>5$.

		On the image of $S$ the operator $\Delta$ coincides with $S^6$ by \eqref{E:Delta2.S6} and \eqref{E:D1S}.
		For any $r\in\mathbb N_0$ we thus get from \eqref{E:etaS}
		\begin{equation}\label{E:etaS.Delta}
			\eta_S(s)
			=\tr\Bigl(S^{1+2r}(S^6)^{-(s+1+2r)/6}\Bigr)
			=\tr\left(S^{1+2r}\Delta^{-(s+1+2r)/6}\right).
		\end{equation}
		The asymptotic expansion in \eqref{E:SDelta.asym} follows from \cite[Lemma~2.18]{H22}.
		Moreover, since $\ker\Delta\subseteq\ker S$ there exists $c>0$ such that $\tr(S^{1+2r}e^{-t\Delta})=O(e^{-ct})$ as $t\to\infty$.
		Hence, via Mellin transform
		\begin{equation}\label{E:etaS.Mellin}
			\eta_S(s)
			=\frac1{\Gamma\left(\frac{s+1+2r}6\right)}\int_0^\infty t^{(s+1+2r)/6}\tr\left(S^{1+2r}e^{-t\Delta}\right)\frac{dt}t,\qquad\Re s>5.
		\end{equation}
		Moreover,
		\begin{equation}\label{E:eta.S.cont}
			\eta_S(s)=\frac1{\Gamma\left(\frac{s+1+2r}6\right)}\left(\sum_{j=0}^{J-1}\frac{6\cdot a_j(S^{1+2r},\Delta)}{s+(j-10)/2}+h_J(s)\right)
		\end{equation}
		where $h_J(s)$ is holomorphic for $\Re s>(10-J)/2$.
		This provides the analytic continuation of $\eta_S$ with simple poles and residues as in \eqref{E:etaS.res}.
		The special values in \eqref{E:etaS.value} can be read off using~\eqref{E:Gamma.res}.
	\end{proof}

	\begin{proposition}\label{P:eta.S.var}
		Suppose $\mathcal D_u$ is a smooth family of (2,3,5) distributions on $M$ and 
		suppose $g_u$ is a smooth family of fiberwise graded Euclidean inner products on the corresponding smooth family of bundles of osculating algebras $\mathfrak t_uM$ where $u\in\mathbb R$.
		Let $S_u$ and $\Delta_u$ denote the operators associated to this data, an orientation of $M$, 
		and a unitary flat complex vector bundle $F$, cf.~\eqref{E:S} and \eqref{E:Delta2.def}.
		Moreover, put $\dot S_u:=\frac\partial{\partial u}S_u$ with respect to a smooth trivialization of the family of bundles of osculating algebras, 
		$\bigsqcup_{u\in\mathbb R}\mathfrak t_uM\cong\mathfrak tM\times\mathbb R$.
		Then
		\begin{equation}\label{E:etaS.dot}
			\tfrac\partial{\partial u}\eta_{S_u}(s)=-s\cdot\tr\left(\dot S_u\Delta_u^{-(s+1)/6}\right).
		\end{equation}
		Moreover,  
		\begin{equation}\label{E:etaS.var}
			\tfrac\partial{\partial u}\eta_{S_u}(s)=-6\cdot\frac{a_{10}(\dot S_u,\Delta_u)}{\Gamma\left(\frac16\right)}+O(s)\quad\text{as $s\to0$}
		\end{equation}
		where $a_j(\dot S_u,\Delta_u)$ denote the coefficients in the asymptotic expansion 
		\begin{equation}\label{E:Sdot.asymp}
			\tr\bigl(\dot S_ue^{-t\Delta_u}\bigr)\sim\sum_{j=0}^\infty t^{(j-12)/12}a_j(\dot S_u,\Delta_u),\qquad\text{as $t\to0$}
		\end{equation}
		and $a_j(\dot S_u,\Delta_u)=0$ for odd $j$.
	\end{proposition}

	\begin{proof}
		Specializing \eqref{E:etaS.Delta} and \eqref{E:etaS.Mellin} to $r=0$ we obtain
		\[
			\eta_{S_u}(s)
			=\tr\left(S_u\Delta_u^{-(s+1)/6}\right)
			=\frac1{\Gamma\left(\frac{s+1}6\right)}\int_0^\infty t^{(s+1)/6}\tr\left(S_ue^{-t\Delta_u}\right)\frac{dt}t,
			\quad\Re s>5.
		\]
		The Hodge decomposition in \cite[Cor~2.16]{DH22} and \eqref{E:HRumin} yield
		\[
			\ker\Delta_u=\ker D_{1,u}^*\cap\ker D_{2,u}=\frac{\ker D_{2,u}}{\img D_{1,u}}=H^2(M;F).
		\]
		Hence, the dimension of $\ker\Delta_u$ is constant in $u$.
		Note that $\Delta_u$ and $S_u$ commute, cf.~\eqref{E:Delta2.S6} and \eqref{E:D1S}. 
		Proceeding as in the proof of Proposition~\ref{P:dot.A} we use Duhamel's formula to obtain
		\begin{equation}\label{E:qwerty1}
			\tfrac\partial{\partial u}\eta_{S_u}(s)
			=\tr\left(\dot S_u\Delta_u^{-(s+1)/6}\right)-\frac{s+1}6\cdot\tr\left(S_u\dot\Delta_u\Delta_u^{-(s+1)/6-1}\right),
			\quad\Re s>5
		\end{equation}
		where $\dot\Delta_u:=\frac\partial{\partial u}\Delta_u$.
		Putting $B_u:=D_{u,1}D_{u,1}^*$ we have, cf.~\eqref{E:Delta2.S6}, 
		\begin{equation}\label{E:Deltau.SB}
			\Delta_u=B_u^2+S_u^6.
		\end{equation}
		Differentiating we obtain
		\[
			\dot\Delta_u=\dot B_uB_u+B_u\dot B_u+\sum_{k=0}^5S_u^k\dot S_uS_u^{5-k}
		\]
		where $\dot B_u:=\frac\partial{\partial u}B_u$.
		Plugging this into \eqref{E:qwerty1}, observing that $B_u$ commutes with $\Delta_u$, and using $B_sS_u=0=S_uB_u$, cf.~\eqref{E:D1S}, we obtain
		\begin{equation}\label{E:qwerty2}
			\tfrac\partial{\partial u}\eta_{S_u}(s)
			=\tr\left(\dot S_u\Delta_u^{-(s+1)/6}\right)-(s+1)\cdot\tr\left(\dot S_uS_u^6\Delta_u^{-(s+1)/6-1}\right),
			\quad\Re s>5.
		\end{equation}
		Differentiating $S_uB_u=0$ we obtain $\dot S_uB_u+S_u\dot B_u=0$ and then $B_u\dot S_uB_u=0$.
		Combining this with \eqref{E:Deltau.SB} and \eqref{E:qwerty2} we obtain \eqref{E:etaS.dot} for $\Re s>5$.
		By analyticity this equation remains true for all $s$.

		The asymptotic expansion in \eqref{E:Sdot.asymp} follows from \cite[Lemma~2.18]{H22}.
		Let $\Pi_u$ denote the orthogonal projection onto $\ker\Delta_u$.
		Note that $\Pi_u$ depends smoothly on $u$ since $\ker\Delta_u$ has constant dimension.
		Differentiating $S_u\Pi_u=0$ we obtain $\dot S_u\Pi_u+S_u\dot\Pi_u=0$ and then $\Pi_u\dot S_u\Pi_u=0$, for we also have $\Pi_uS_u=0$ by symmetry of $S_u$.
		Hence, there exists $c>0$ such that $\tr\bigl(\dot S_ue^{-t\Delta_u}\bigr)=O(e^{-ct})$ as $t\to\infty$.
		From \eqref{E:etaS.dot} we have
		\begin{equation}\label{E:ddu.eta.Su}
			\tfrac\partial{\partial u}\eta_{S_u}(s)=\frac{-s}{\Gamma\left(\frac{s+1}6\right)}\int_0^\infty t^{(s+1)/6}\tr\left(\dot S_ue^{-t\Delta_u}\right)\frac{dt}t.
		\end{equation}
		Combining this with the asymptotic expansion in \eqref{E:Sdot.asymp} we get
		\begin{equation}\label{E:ddu.eta.Su.cont}
			\tfrac\partial{\partial u}\eta_{S_u}(s)
			=\frac{-s}{\Gamma\left(\frac{s+1}6\right)}\left(\sum_{j=0}^{J-1}\frac{6\cdot a_j(\dot S_u,\Delta_u)}{s+(j-10)/2}+h_{J,u}(s)\right)
		\end{equation}
		where $h_{J,u}(s)$ is holomorphic for $\Re s>(10-J)/2$.
		Using $J\geq11$ we get \eqref{E:etaS.var}.
	\end{proof}

	\begin{proposition}\label{P:loc.quant}
		The coefficients $a_j(S^{1+2r},\Delta)$ in \eqref{E:SDelta.asym} and $a_j(\dot S_u,\Delta_u)$ in \eqref{E:Sdot.asymp}  all vanish.
	\end{proposition}

	\begin{proof}
		Let $S_\triv=(\mathbf i\star_3D_2)^{F_\triv}$ denote the operator associated with the trivial flat vector bundle $F_\triv=M\times\mathbb C^{\rk F}$, cf.~\eqref{E:S}.
		As the untwisted operator $\star_3D_2$ is skew-adjoint and real, the spectrum of $S_\triv$ is symmetrical about zero, cf.~\cite[Remark~3a, p~61]{APS75}.
		Hence, the eta function of $S_\triv$ vanishes identically, cf.~\eqref{E:etaS.intro}.
		In view of \eqref{E:etaS.res} and \eqref{E:etaS.value} the coefficients $a_j(S^{1+2r},\Delta)$ must all vanish for $F_\triv$.
		\footnote{
			Alternatively, the latter can be read of \eqref{E:eta.S.cont}. It may also be deduce directly by observing that 
			$\tr(S^{1+2r}e^{-t\Delta})$ vanishes identically as the spectrum of $S^{1+2r}e^{-t\Delta}$ is symmetrical about zero too, cf.~\eqref{E:D1S} and \eqref{E:spec.ASD}.
		}
		Locally, $S=(\mathbf i\star_3D_2)^F$ and $S_\triv$ are the same.
		As the coefficients $a_j(S^{1+2r},\Delta)$ can be computed locally, they must be the same for $F$ and $F_\triv$.
		Hence, $a_j(S^{1+2r},\Delta)=0$ for all unitary flat bundles $F$.
		The vanishing of $a_j(\dot S_u,\Delta_u)$ can be proved along the same lines using \eqref{E:ddu.eta.Su.cont}.
	\end{proof}

	Combining Propositions~\ref{P:etaSs}, \ref{P:eta.S.var} and \ref{P:loc.quant} we obtain Theorem~\ref{T:eta.S.hinv}.
	In particular, the eta function $\eta_S(s)$ is regular at $s=0$ and the eta invariant of $S$ may be defined by $\eta(S)=\eta_S(0)$, cf.~\eqref{E:eta.S.def.intro}.

	\begin{remark}\label{E:etaS.ori}
		Reversing the orientation of $M$ changes $\star_3$ by a factor of $-1$.
		Hence, $S$, $\eta_S(s)$, and $\eta(S)$ also change by a factor of $-1$ if the orientation is reversed, cf.~\eqref{E:S}, \eqref{E:etaS}, and \eqref{E:eta.S.def.intro}.
	\end{remark}

\section{The eta invariant in irreducible unitary representations}\label{S:reps}

	Recall that $G$ denotes the simply connected nilpotent Lie group with Lie algebra $\goe$, cf.~\eqref{E:goe.intro}.
	Let $X_1,\dotsc,X_5$ be a graded basis of $\goe$ satisfying the relations in \eqref{E:brackets.intro}.
	In this section we use harmonic analysis on $G$ to study the operator $S$ associated with the left invariant (2,3,5) distribution $\mathcal D_G$ spanned by the vector fields $X_1$ and $X_2$.
	More precisely, we will study the eta function of $S$ in irreducible unitary representations of $G$.

	Using left translation to trivialize the bundle of osculating algebras, $\mathfrak tG=G\times\goe$, and the bundle $\mathcal H^q(\mathfrak tG)=G\times H^q(\goe)$, 
	we may consider the corresponding (left invariant) Rumin differential as
	\[
		D_2\in\mathcal U^{-2}(\goe)\otimes L\bigl(H^2(\goe),H^3(\goe)\bigr).
	\]
	Here $H^q(\goe)$ denotes the Lie algebra cohomology of $\goe$ with trivial real coefficients 
	and the grading on the universal enveloping algebra $\mathcal U(\goe)$ is induced from the grading automorphism of $\goe$.
	With respect to a particular basis of $H^q(\goe)$ described in \cite[\S3.3]{H23a} we have
	\begin{equation}\label{E:D2}
		D_2=\begin{pmatrix}-X_1X_2-X_3&X_1X_1/\sqrt2&0\\-X_2X_2/\sqrt2&-\tfrac32X_3&X_1X_1/\sqrt2\\0&-X_2X_2/\sqrt2&X_2X_1-X_3\end{pmatrix}
	\end{equation}
	where $X_1,\dotsc,X_5$ are considered as left invariant differential operators on $G$.

	Let $g$ be a graded Euclidean inner product on $\goe$ and let $\mathfrak o$ be an orientation of $\goe$.
	Note that the Lie brackets in~\eqref{E:brackets.intro} provide natural one-to-one correspondences between 
	orientations of $\goe_{-1}$, orientations of $\goe_{-2}$, orientations of $\goe_{-3}$, and orientations of $\goe$.
	Recall that $g_G$ denotes the corresponding left invariant fiber wise graded Euclidean inner product on $\mathfrak tG$ and $\mathfrak o_G$ denotes the corresponding orientation of $G$.
	Via the aforementioned left trivialization $\mathcal H^q(\mathfrak tG)=G\times H^q(\goe)$ 
	the associated Hodge star operator $\star_3=\star^G_{\mathcal D_G,g_G,\mathfrak o_G,3}\colon\mathcal H^3(\mathfrak tG)\to\mathcal H^2(\mathfrak tG)$
	coincides with the Hodge star operator $\star_3=\star^\goe_{g,\mathfrak o,3}\colon H^3(\goe)\to H^2(\goe)$ associated with $g$ and $\mathfrak o$.
	Hence, we may consider, cf.~\eqref{E:S},
	\begin{equation}\label{E:SG.def}
		S=S^G_{g,\mathfrak o}=S^{G,\mathbb C}_{\mathcal D_G,g_G,\mathfrak o_G}=\mathbf i\star_3D_2\in\mathcal U^{-2}(\goe)\otimes\eend\bigl(H^2(\goe)\otimes\mathbb C\bigr),
	\end{equation}
	and, cf.~\eqref{E:Delta2.def},
	\begin{equation}\label{E:DG.def}
		\Delta=\Delta^G_g=\Delta^{G,\mathbb C}_{\mathcal D_G,g_G}=(D_1D_1^*)^2+(D_2^*D_2)^3\in\mathcal U^{-12}(\goe)\otimes\eend\bigl(H^2(\goe)\otimes\mathbb C\bigr).
	\end{equation}

	From \eqref{E:brackets.intro} we see that restriction provides an isomorphism
	\begin{equation}\label{E:Aut.gr}
		\Aut_\grr(\goe)=\GL(\goe_{-1})\cong\GL_2(\mathbb R)
	\end{equation}
	that also identifies the orientation preserving subgroups, $\Aut^+_\grr(\goe)=\GL^+(\goe_{-1})$.
	If $\phi\in\Aut_\grr(\goe)$ is a graded automorphism then, by naturality,
	\[
		\phi\cdot S^G_{g,\mathfrak o}=S^G_{\phi\cdot g,\phi\cdot\mathfrak o}
	\]
	where the dots denote the natural left actions.
	Hence,
	\begin{equation}\label{E:rho.S.phi}
		(\rho\circ\phi)\bigl(S^G_{g,\mathfrak o}\bigr)=\rho\bigl(S^G_{\phi\cdot g,\phi\cdot\mathfrak o}\bigr)
	\end{equation}
	for every unitary representation $\rho$ of $G$ on a Hilbert space $\mathcal H$. 
	
	Put $g_{ij}:=g(X_i,X_j)$.
	Up to a graded isomorphism of $\goe$ we may assume that $X_1,X_2$ is a positively oriented orthonormal basis of $\goe_{-1}$, cf.~\eqref{E:Aut.gr}.
	By the principal axis theorem we may, moreover, assume that $b_g$ in \eqref{E:bg} is diagonal in this basis, i.e., $b_g(X_1,X_2)=0$, equivalently, $g(X_4,X_5)=0$.
	In this case $a_g=g_{33}$, cf.~\eqref{E:ag}.
	Moreover, the Euclidean inner products $g$ and $b_g$ are represented by the matrices
	\begin{equation}\label{E:gbg}
		g=\begin{pmatrix}1\\&1\\&&g_{33}\\&&&g_{44}\\&&&&g_{55}\end{pmatrix}
		\qquad\text{and}\qquad
		b_g=\frac1{g_{33}}\begin{pmatrix}g_{44}\\&g_{55}\end{pmatrix}
	\end{equation}
	with respect to the basis $X_1,\dotsc,X_5$ of $\goe$ and $X_1,X_2$ of $\goe_{-1}$, respectively.
	With respect to the aforementioned basis of $H^q(\goe)$ the induced Hermitian inner product $h_2$ on $H^2(\goe)\otimes\mathbb C$ is
	\[
		h_2
		=\frac1{\sqrt{g_{44}g_{55}}}
		\begin{pmatrix}
			\sqrt{g_{55}/g_{44}}\\
			&\frac{(g_{44}+g_{55})/2}{\sqrt{g_{44}g_{55}}}\\
			&&\sqrt{g_{44}/g_{55}}
		\end{pmatrix}
	\]
	see \cite[\S3.3]{H23a}, and the Hodge star operator $\star_3\colon H^3(\goe)\to H^2(\goe)$ is
	\begin{equation}\label{E:star3}
		\star_3=\star_2^{-1}=\star_2^*=
		\frac1{\sqrt{g_{33}}}\begin{pmatrix}&&\sqrt{g_{44}/g_{55}}\\&-\frac{\sqrt{g_{44}g_{55}}}{(g_{44}+g_{55})/2}\\\sqrt{g_{55}/g_{44}}\end{pmatrix}
	\end{equation}
	see \cite[\S3.4]{H23a}.

	For an irreducible unitary representation $\rho$ of $G$ on a Hilbert space $\mathcal H$ put
	\begin{equation}\label{E:eta.rho.def}
		\eta_{\rho(S)}(s):=\tr\left(\rho(S)|\rho(S)|^{-s-1}\right)=\sum_{\lambda\in\spec_*(\rho(S))}\sign(\lambda)|\lambda|^{-s}
	\end{equation}
	and
	\[
		\eta(\rho(S)):=\eta_{\rho(S)}(0).
	\]
	Below we will show that \eqref{E:eta.rho.def} converges for $\Re s$ sufficiently large, depending on the representation $\rho$, 
	and that $\eta_{\rho(S)}(s)$ extends to a meromorphic function on the entire complex plane which is holomorphic at $s=0$.
	This will be accomplished using the Rockland operator $\Delta$ in \eqref{E:Delta2.def}.
	As in \eqref{E:etaS.Delta} and \eqref{E:etaS.Mellin}, for $\Re s $ sufficiently large,
	\begin{align}\notag
		\eta_{\rho(S)}(s)
		&=\tr\left(\rho(S)\rho(\Delta)^{-(s+1)/6}\right)
		\\\label{E:etaS.Delta.rho}
		&=\frac1{\Gamma\left(\frac{s+1}6\right)}\int_0^\infty t^{(s+1)/6}\tr\left(\rho(S)e^{-t\rho(\Delta)}\right)\frac{dt}t.
	\end{align}
	The asymptotic expansion of the trace in the integrand, to be established in Lemma~\ref{L:asym.rho} below, will yield the analytic continuation of $\eta_{\rho(S)}$.
	
	The unitary dual of the Lie group $G$ has been described explicitly by Dixmier in \cite[Prop~8]{D58}. 
	There are three types of irreducible (strongly continuous) unitary representations of $G$ 
	which we now describe in terms of a graded basis $X_1,\dotsc,X_5$ of $\goe$ satisfying \eqref{E:brackets.intro}.
	Notationally, we do not distinguish between a representation of $G$ and the corresponding infinitesimal representation of $\goe$.

	\begin{enumerate}[(I)]
		\item	\emph{Scalar representations:}
			For $(\alpha,\beta)\in\mathbb R^2$ there is an irreducible unitary representation $\rho_{\alpha,\beta}$ of $G$ on $\mathbb C$ such that
			\begin{equation}\label{E:rep.scalar}
				\rho_{\alpha,\beta}(X_1)=2\pi\mathbf i\alpha,\qquad
				\rho_{\alpha,\beta}(X_2)=2\pi\mathbf i\beta,
			\end{equation}
			and $\rho_{\alpha,\beta}(X_3)=\rho_{\alpha,\beta}(X_4)=\rho_{\alpha,\beta}(X_5)=0$.
			These are the irreducible unitary representations which factor through the abelianization $G/[G,G]\cong\mathbb R^2$.
		\item	\emph{Schr\"odinger representations:}
			For $0\neq\hbar\in\mathbb R$ there is an irreducible unitary representation $\rho_\hbar$ of $G$ on $L^2(\mathbb R)=L^2(\mathbb R,d\theta)$ such that
			\begin{equation}\label{E:Schroedinger}
				\rho_\hbar(X_1)=\tfrac\partial{\partial\theta},\qquad
				\rho_\hbar(X_2)=2\pi\mathbf i\hbar\cdot\theta,\qquad
				\rho_\hbar(X_3)=2\pi\mathbf i\hbar,
			\end{equation}
			and $\rho_\hbar(X_4)=\rho_\hbar(X_5)=0$.
			These are the irreducible unitary representations which factor through the 3-dimensional Heisenberg group $H=G/C$ but do not factor through the abelianization $G/[G,G]$.
		\item	\emph{Generic representations:}
			For real numbers $\lambda,\mu,\nu$ with $(\lambda,\mu)\neq(0,0)$ there is an irreducible unitary representation 
			$\rho_{\lambda,\mu,\nu}$ of $G$ on $L^2(\mathbb R)=L^2(\mathbb R,d\theta)$ such that:
			\begin{align}
				\notag
				\rho_{\lambda,\mu,\nu}(X_1)
				&=\frac\lambda{(\lambda^2+\mu^2)^{1/3}}\cdot\partial_\theta-\frac{2\pi\mathbf i\mu}{(\lambda^2+\mu^2)^{1/3}}\cdot\frac{\theta^2+\nu(\lambda^2+\mu^2)^{-2/3}}2
				\\\notag
				\rho_{\lambda,\mu,\nu}(X_2)
				&=\frac\mu{(\lambda^2+\mu^2)^{1/3}}\cdot\partial_\theta+\frac{2\pi\mathbf i\lambda}{(\lambda^2+\mu^2)^{1/3}}\cdot\frac{\theta^2+\nu(\lambda^2+\mu^2)^{-2/3}}2
				\\\label{E:rep.gen.X3}
				\rho_{\lambda,\mu,\nu}(X_3)
				&=2\pi\mathbf i(\lambda^2+\mu^2)^{1/3}\cdot\theta,
				\\\notag
				\rho_{\lambda,\mu,\nu}(X_4)
				&=2\pi\mathbf i\lambda,
				\\\notag
				\rho_{\lambda,\mu,\nu}(X_5)
				&=2\pi\mathbf i\mu.
			\end{align}
			This differs from the representation given in \cite[Eq~(24)]{D58} by a conjugation with a unitary scaling on $L^2(\mathbb R)$ 
			which we have introduced for better compatibility with the grading automorphism. 
	\end{enumerate}
	The representations listed above are mutually nonequivalent, and they comprise all equivalence classes of irreducible unitary representations of $G$.

	\begin{lemma}\label{L:asym.rho}
		Let $V$ finite dimensional complex vector space equipped with a Hermitian inner product, and suppose $A=A^*\in\mathcal U^{-2\kappa}(\goe)\otimes\eend(V)$ is 
		a left invariant positive Rockland differential operator on the 5-dimensional Lie group $G$ which is homogeneous of order $2\kappa>0$ with respect to the grading automorphism.
		Then, for every left invariant differential operator which is homogeneous of order $r\geq0$ with respect to the grading automorphism, 
		$B\in\mathcal U^{-r}(\goe)\otimes\eend(V)$, the following hold true:
		\begin{enumerate}[(I)]
			\item	For all nontrivial irreducible unitary representations $\rho$ of $G$ and all $N$,
				\begin{equation}\label{E:asymp.infty}
					\tr\left(\rho(B)e^{-t\rho(A)}\right)=O(t^{-N}),
				\end{equation}
				as $t\to\infty$.
			\item	In the Schr\"odinger representation $\rho_\hbar$,
				\begin{equation}\label{E:asymp.Schroedinger}
					\tr\left(\rho_\hbar(B)e^{-t\rho_\hbar(A)}\right)
					\sim t^{-r/2\kappa}\sum_{j=0}^\infty t^{(j-1)/\kappa}d_j,
				\end{equation}
				as $t\to0$ with coefficients $d_j=d_{A,B,\hbar,j}$.
			\item	In the generic representation $\rho_{\lambda,\mu,\nu}$ we have
				\begin{equation}\label{E:asymp.gen}
					\tr\left(\rho_{\lambda,\mu,\nu}(B)e^{-t\rho_{\lambda,\mu,\nu}(A)}\right)
					\sim t^{-r/2\kappa}\sum_{j=0}^\infty t^{(j-3)/4\kappa}a_j,
				\end{equation}
				as $t\to0$.
				Moreover, the coefficients $a_j=a_{A,B,\lambda,\mu,\nu,j}$ vanish for odd $j$.
		\end{enumerate}
	\end{lemma}

	\begin{proof}
		The case $r=0$ has been established in \cite[Thm~5]{H23a}.
		The proof there readily generalizes to $r\geq0$.
	\end{proof}

	\begin{lemma}\label{L:eta.scalar}
		Suppose $\rho=\rho_{\alpha,\beta}$ is an irreducible unitary scalar representation of $G$.
		Then, with respect to any graded Euclidean inner product $g$ on $\goe$ and either orientation $\mathfrak o$ of $\goe$, 
		we have $\eta_{\rho(S)}(s)=0$ and $\eta(\rho(S))=0$ where $S=S^G_{g,\mathfrak o}$, cf.~\eqref{E:SG.def}.
	\end{lemma}

	\begin{proof}
		Combining \eqref{E:D2} and \eqref{E:rep.scalar} we get
		\[
			\rho(D_2)=-4\pi^2\begin{pmatrix}-\alpha\beta&\alpha^2/\sqrt2&0\\-\beta^2/\sqrt2&0&\alpha^2/\sqrt2\\0&-\beta^2/\sqrt2&\alpha\beta\end{pmatrix}.
		\]
		In view of \eqref{E:Aut.gr} and \eqref{E:rho.S.phi} we may assume that the metric takes the form in \eqref{E:gbg} and that $X_1,\dotsc,X_5$ is positively oriented. 
		Using \eqref{E:star3} and \eqref{E:SG.def} we obtain
		\[
			\rho(S)=\frac{4\pi^2\mathbf i}{\sqrt{g_{33}}}
			\begin{pmatrix}
				0&\sqrt{\frac{g_{44}}{g_{55}}}\cdot\frac{\beta^2}{\sqrt2}&-\sqrt{\frac{g_{44}}{g_{55}}}\cdot\alpha\beta\\
				-\frac{\sqrt{g_{44}g_{55}}}{(g_{44}+g_{55})/2}\cdot\frac{\beta^2}{\sqrt2}&0&\frac{\sqrt{g_{44}g_{55}}}{(g_{44}+g_{55})/2}\cdot\frac{\alpha^2}{\sqrt2}\\
				\sqrt{\frac{g_{55}}{g_{44}}}\cdot\alpha\beta&-\sqrt{\frac{g_{55}}{g_{44}}}\cdot\frac{\alpha^2}{\sqrt2}&0
			\end{pmatrix}.
		\]
		This matrix has vanishing trace and determinant.
		Hence, its spectrum is symmetrical about zero.
		Therefore, $\eta_{\rho(S)}(s)=0$ and $\eta(\rho(S))=0$.
	\end{proof}

	Before we move on to Schr\"odinger and generic representations we give a 

	\begin{proof}[{Proof of Lemma~\ref{L:tetaa}}]
		For $m=0,1,2,\dotsc$ and $a>-\lambda_m$ we put 
		\[
			\zeta_m(s,a)=\sum_{n=m}^\infty(\lambda_n+a)^{-s}
		\]
		with $\lambda_n$ as in \eqref{E:lambdan}.
		Then, cf.~\eqref{E:tetaa.def},
		\begin{multline}\label{E:tetam}
			\tilde\eta(s,a)
			=\sum_{n=0}^{m-1}\sign(a+\lambda_n)|a+\lambda_n|^{-s}
			+\sum_{n=0}^{m-1}\sign(a-\lambda_n)|a-\lambda_n|^{-s}
			\\+\zeta_m(s,a)-\zeta_m(s,-a)
		\end{multline}
		for $|a|<\lambda_m$.

		Putting 
		\begin{equation}\label{E:zetam.def}
			\zeta_m(s):=\zeta_m(s,0)=\sum_{n=m}^\infty\lambda_n^{-s}
		\end{equation}
		we have, for $a>-\lambda_m$ and $L\in\mathbb N_0$,
		\begin{equation}\label{E:zetasa}
			\zeta_m(s,a)=\sum_{l=0}^L\binom{-s}la^l\zeta_m(s+l)+h_{m,a,L}(s)
		\end{equation}
		where $h_{m,a,L}(s)$ is holomorphic for $\Re s>-L$ and 
		\begin{equation}\label{E:ha.values}
			h_{m,a,L}(0)=h_{m,a,L}(-1)=\cdots=h_{m,a,L}(-L+1)=0.
		\end{equation}
		Indeed, $\lambda_n=O(n)$ as $n\to\infty$ and 
		\begin{multline*}
			h_{m,a,L}(s)
			=\zeta_m(s,a)-\sum_{l=0}^L\binom{-s}la^l\zeta_m(s+l)
			\\
			=\sum_{n=m}^\infty\lambda_n^{-s}\underbrace{\left(\left(1+\frac a{\lambda_n}\right)^{-s}
			-\sum_{l=0}^L\binom{-s}l\left(\frac a{\lambda_n}\right)^l\right)}_{O\bigl(s(s+1)\cdots(s+L)\cdot\lambda_n^{-L-1}\bigr)}
		\end{multline*}
		where the latter estimate follows from
		\begin{equation}\label{E:zs}
			(1+z)^{-s}=\sum_{l=0}^L\binom{-s}lz^l+O\Bigl(s(s+1)\cdots(s+L)z^{L+1}\Bigr)
		\end{equation}
		as $z\to0$, uniformly for $s$ in compact subsets.

		Furthermore,
		\begin{equation}\label{E:zetas}
			\zeta_0(s)=2^{s/2}\sum_{l=0}^L\binom{-s/2}l\left(\frac98\right)^l\bigl(1-2^{-s-2l}\bigr)\zeta_\Riem(s+2l)+h_L(s)
		\end{equation}
		where $\zeta_\Riem(s)=\sum_{n=1}^\infty n^{-s}$ denotes the Riemann zeta function, $h_L(s)$ is holomorphic for $\Re s>-2L-1$ and
		\begin{equation}\label{E:h.values}
			h_L(0)=h_L(-2)=\cdots=h_L(-2L)=0.
		\end{equation}
		Indeed,
		\begin{multline*}
			h_L(s)
			=\zeta_0(s)-2^{s/2}\sum_{l=0}^L\binom{-s/2}l\left(\frac98\right)^l\bigl(1-2^{-s-2l}\bigr)\zeta_\Riem(s+2l)
			\\
			=2^{s/2}\sum_{n=0}^\infty(2n+1)^{-s}\underbrace{\left(\left(\frac{\sqrt2\cdot\lambda_n}{2n+1}\right)^{-s}
			-\sum_{l=0}^L\binom{-s/2}l\left(\frac98\right)^l(2n+1)^{-2l}\right)}_{O\bigl(s(s+2)\cdots(s+2L)\cdot(2n+1)^{-2L-2}\bigr)}
		\end{multline*}
		where the estimate follows from \eqref{E:zs} and
		\[
			\left(\frac{\sqrt2\cdot\lambda_n}{2n+1}\right)^{-s}
			=\left(1+\frac98(2n+1)^{-2}\right)^{-s/2}.
		\]

		Clearly, cf.~\eqref{E:zetam.def},
		\begin{equation}\label{E:zetam.trunc}
			\zeta_m(s)=\zeta_0(s)-\sum_{n=0}^{m-1}\lambda_n^{-s}.
		\end{equation}
		Since the Riemann zeta function has a single simple pole at $s=1$ with residue $1$, we conclude from \eqref{E:zetas} and \eqref{E:zetam.trunc} that $\zeta_m(s)$
		extends to a meromorphic function on the entire complex plane with simple poles located at $s=-2l+1$, $l\in\mathbb N_0$ and residues
		\begin{equation}\label{E:zeta.res}
			\res_{s=-2l+1}\zeta_m(s)
			=\frac{1\cdot3\cdot5\cdots(2l-1)}{\sqrt2\cdot4^l\cdot l!}\left(\frac98\right)^l
			=\frac1{\sqrt2}\binom{2l}l\left(\frac9{8\cdot8}\right)^l.
		\end{equation}
		As the Riemann zeta function vanishes at negative even integers, \eqref{E:zetas} and \eqref{E:h.values} moreover yield $\zeta_0(-2l)=0$, i.e.,
		\begin{equation}\label{E:zeta.values}
			\zeta_m(-2l)=-\sum_{n=0}^{m-1}\lambda_n^{2l},\qquad l=0,1,2,\dotsc
		\end{equation}

		Combining \eqref{E:zeta.res} with \eqref{E:zetasa} and \eqref{E:ha.values} we see that for $|a|<\lambda_m$ the function
		\begin{multline*}
			\zeta_m(s,a)-\zeta_m(s,-a)
			=2\sum_{j=0}^J\binom{-s}{2j+1}a^{2j+1}\zeta_m(s+2j+1)
			\\
			+h_{m,a,2J+1}(s)-h_{m,-a,2J+1}(s)
		\end{multline*}
		extends to a meromorphic function on the entire complex plane with simple poles located at $s=-2l$ where $l=1,2,\dotsc$ and residues
		\[
			\res_{s=-2l}\zeta_m(s,a)-\zeta_m(s,-a)
			=\sqrt2\cdot\sum_{j=0}^{l-1}\binom{2l}{2j+1}\binom{2l-2j}{l-j}\left(\frac9{8\cdot8}\right)^{l-j}a^{2j+1}
		\]
		as well as
		\[
			\zeta_m(0,a)-\zeta_m(0,-a)
			=-\sqrt2\cdot a.
		\]
		Moreover, \eqref{E:zeta.values} and the binomial theorem yield
		\[
			\zeta_m(-2l-1,a)-\zeta_m(-2l-1,-a)
			=-\sum_{n=0}^{m-1}(a+\lambda_n)^{2l+1}-\sum_{n=0}^{m-1}(a-\lambda_n)^{2l+1}
		\]
		for $l=0,1,2,\dotsc$
		Combining these with \eqref{E:tetam} we obtain Lemma~\ref{L:tetaa}.
	\end{proof}

	\begin{lemma}\label{L:eta.Schroedinger}
		Suppose $\rho=\rho_\hbar$ is a Schr\"odinger representation of $G$,
		let $g$ be a graded Euclidean inner product on $\goe$,
		let $\mathfrak o$ be an orientation of $\goe$, 
		and consider the operator $S=S^G_{g,\mathfrak o}$, cf.~\eqref{E:SG.def}.

		(a)
		Then the right hand side of \eqref{E:eta.rho.def} converges absolutely for $\Re s>1$ and $\eta_{\rho(S)}(s)$ 
		extends to a meromorphic function on the entire complex plane whose poles are all simple and can only be located at $s=1-j$ with $1\neq j\in\mathbb N_0$.
		In particular, $\eta_{\rho(S)}(s)$ is holomorphic at $s=0$.

		(b) If, moreover, $b_g$ is proportional to $g|_{\goe_{-1}}$, then 
		\begin{equation}\label{E:etas}
			\eta_{\rho(S)}(s)=\frac{\sign(\hbar)}{\mathfrak o(X_3)}\cdot\left(\frac{2\pi|\hbar|}{\sqrt{g_{33}}}\right)^{-s}\cdot\tilde\eta\bigl(s,\tfrac54\bigr)
		\end{equation}
		where $\tilde\eta(s,a)$ is the function in~\eqref{E:tetaa.def}, $g_{33}=g(X_3,X_3)$, and $\mathfrak o(X_3)=\pm1$ 
		depending on whether $X_3$ is a positively or negatively oriented basis of $\goe_{-2}$ with respect to the orientation induced by $\mathfrak o$.
		In particular, the poles only occur at negative even integers and the residues are
		\begin{multline}\label{E:etaS.rhoh-2}
			\res_{s=-2l}\eta_{\rho(S)}(s)
			=\frac{\sign(\hbar)}{\mathfrak o(X_3)}\cdot\left(\frac{2\pi\hbar}{\sqrt{g_{33}}}\right)^{2l}
			\\
			\cdot\sqrt2\cdot\sum_{j=0}^{l-1}\binom{2l}{2j+1}\binom{2l-2j}{l-j}\left(\frac9{8\cdot8}\right)^{l-j}\left(\frac54\right)^{2j+1}.
		\end{multline}
		Moreover,
		\begin{equation}\label{E:etaS.rhoh}
			\eta(\rho(S))=\eta_{\rho(S)}(0)
			=\frac{\sign(\hbar)}{\mathfrak o(X_3)}\cdot\left(2-\frac{5\sqrt2}4\right)
		\end{equation}
		as well as 
		\begin{equation}\label{E:etaS.rhoh-1}
                        \eta_{\rho(S)}(-2l-1)=0,\qquad l=0,1,2,\dotsc
		\end{equation}
	\end{lemma}

	\begin{proof}
		Using \eqref{E:asymp.infty} and \eqref{E:asymp.Schroedinger} we see that \eqref{E:etaS.Delta.rho} converges for $\Re s>1$ and that
		\[
			\eta_{\rho(S)}(s)
			=\frac1{\Gamma\left(\frac{s+1}6\right)}\sum_{j=0}^{J-1}\frac{6\cdot d_j}{s+j-1}+h_J(s)
		\]
		where $h_J(s)$ is holomorphic for $\Re s>1-J$.
		Hence, $\eta_{\rho(S)}(s)$ extends to a meromorphic function on the entire complex plane, its poles are all simple, and they can only by located at $s=1-j$ with $j\in\mathbb N_0$.
		To complete the proof of (a) we still have to rule out a pole at $s=0$.

		Suppose $g_u$ is a smooth curve of graded Euclidean inner products on $\goe$ and let $S_u$ and $\Delta_u$ denote the associated operator.
		As in Proposition~\ref{P:eta.S.var} we have
		\begin{equation}\label{E:eta.S.dot}
			\tfrac\partial{\partial u}\eta_{\rho(S_u)}(s)=-s\cdot\tr\left(\rho(\dot S_u)\rho(\Delta_u)^{-(s+1)/6}\right)
		\end{equation}
		and, thus, the residue of $\eta_{\rho(S_u)}(s)$ at $s=0$ is independent of $u$.
		Hence, to complete the proof of part~(a) it suffices to show that $\eta_{\rho(S)}(s)$ is regular at $s=0$ for a single graded Euclidean inner product $g$ on $\goe$.
		This will be accomplished during the second part of the proof.

		To prove part~(b) we assume that $b_g$ is proportional to $g|_{\goe_{-1}}$.
		Then the nonzero spectrum of $\rho(S)$ consists of the following eigenvalues, all of which have multiplicity one:
		\begin{equation}\label{E:spec.S.h}
			\spec_*(\rho(S)):\quad\frac{2\pi\hbar}{\mathfrak o(X_3)\sqrt{g_{33}}}\cdot\frac{5\pm\sqrt{8(2n+1)^2+9}}4,\qquad n=0,1,2,3,\dotsc
		\end{equation}
		To see this note first that for every graded automorphism $\phi\in\Aut_\grr(\goe)$ the operators 
		$\rho_\hbar(S^G_{\phi\cdot g,\phi\cdot\mathfrak o})$ and $\rho_{\hbar\cdot\det(\phi|_{\goe_{-1}})}(S^G_{g,\mathfrak o})$ are isospectral by \eqref{E:rho.S.phi} 
		and since we have a unitary equivalence of representations $\rho_\hbar\circ\phi\sim\rho_{\hbar\cdot\det(\phi|_{\goe_{-1}})}$, cf.~\cite[\S3.2]{H23a}.
		In view of \eqref{E:Aut.gr} and since $\phi(X_3)=\det(\phi|_{\goe_{-1}})X_3$ we may therefore assume w.l.o.g.\ that 
		the metric takes the form in \eqref{E:gbg} and that $X_1,\dotsc,X_5$ is positively oriented.
		If $X_1,\dotsc,X_5$ is an orthonormal basis then \eqref{E:spec.S.h} follows from \cite[Lemmas~2 and 3]{H23a}, cf.~\eqref{E:S}, \eqref{E:D2}, and \eqref{E:star3}.
		For more general $g$, such that $b_g$ is proportional to $g$, we have $g_{44}=g_{55}$, see~\eqref{E:gbg}.
		This changes $\star_3$ and, thus, $S$ by a factor of $g_{33}^{-1/2}$, see \eqref{E:star3} and \eqref{E:SG.def}.
		This completes the proof of~\eqref{E:spec.S.h}.
		Using \eqref{E:eta.rho.def} and \eqref{E:tetaa.def} we obtain the formula in \eqref{E:etas}.
		Combining this with Lemma~\ref{L:tetaa} we obtain the special values of $\eta_{\rho(S)}(s)$ stated in \eqref{E:etaS.rhoh-2}--\eqref{E:etaS.rhoh-1}.
		This concludes the proof of part~(b).
		We also conclude that $\eta_{\rho(S)}(s)$ is regular at $s=0$, thus also completing the proof of part~(a).
	\end{proof}

	\begin{lemma}\label{L:eta.generic}
		Suppose $\rho=\rho_{\lambda.\mu,\nu}$ is a generic representation of $G$,
		let $g$ be a graded Euclidean inner product on $\goe$,
		let $\mathfrak o$ be an orientation of $\goe$, 
		and consider the operator $S=S^G_{g,\mathfrak o}$, cf.~\eqref{E:SG.def}.

		(a)
		Then the right hand side of \eqref{E:eta.rho.def} converges absolutely for $\Re s>3/4$ and $\eta_{\rho(S)}(s)$ 
		extends to a meromorphic function on the entire complex plane whose poles are all simple and can only be located at $s=(3-2j)/4$ with $j\in\mathbb N_0$.
		In particular, $\eta_{\rho(S)}(s)$ is holomorphic at $s=0$.
		Moreover, this eta function vanishes at $s=0$, that is,
		\[
			\eta(\rho(S))=0.
		\]

		(b) If, moreover, $b_g$ is proportional to $g|_{\goe_{-1}}$, then $\eta_{\rho(S)}(s)=0$ for all $s$.
	\end{lemma}

	\begin{proof}
		Using \eqref{E:asymp.infty} and \eqref{E:asymp.gen} we see that \eqref{E:etaS.Delta.rho} converges for $\Re s>3/4$ and that
		\[
			\eta_{\rho(S)}(s)
			=\frac1{\Gamma\left(\frac{s+1}6\right)}\sum_{j=0}^{J-1}\frac{6\cdot a_j}{s+(j-3)/4}+h_J(s)
		\]
		where $h_J(s)$ is holomorphic for $\Re s>(3-J)/4$.
		Moreover, $a_j=0$ for odd $j$.
		Hence, $\eta_{\rho(S)}(s)$ extends to a meromorphic function on the entire complex plane, its poles are all simple, and they can only by located at $s=(3-2j)/4$ with $j\in\mathbb N_0$.
		In particular, $\eta_{\rho(S)}(s)$ is regular at $s=0$.
		To complete the proof of part~(a) it remains to show $\eta_{\rho(S)}(0)=0$.

		We will next show that $\eta_{\rho(S)}(0)$ is independent of $g$.
		To this end, suppose $g_u$ is a smooth curve of graded Euclidean metrics on $\goe$ and let $S_u$ and $\Delta_u$ denote the associated operators.
		As in Proposition~\ref{P:eta.S.var}, cf.~\eqref{E:ddu.eta.Su}, we have
		\begin{equation}\label{E:eta.dot.rho}
			\tfrac\partial{\partial u}\eta_{\rho(S_u)}(s)
			=\frac{-s}{\Gamma\left(\frac{s+1}6\right)}\int_0^\infty t^{(s+1)/6}\tr\left(\rho(\dot S_u)e^{-t\rho(\Delta_u)}\right)\frac{dt}t
		\end{equation}
		where $\dot S_u:=\frac\partial{\partial u}S_u$.
		Using the asymptotic expansion in Lemma~\ref{L:asym.rho}(III) we conclude
		\begin{equation}\label{E:Avarg}
			\tfrac\partial{\partial u}\eta_{\rho(S_u)}(0)=0.
		\end{equation}
		This shows that $\eta_{\rho(S)}(0)$ does not depend on the Euclidean metric $g$.
		Hence, in order to complete the proof of part~(a) it suffices to show $\eta_{\rho(S)}(0)=0$ for a single graded Euclidean inner product $g$ on $\goe$.
		This will be accomplished during the second part of the proof.

		To prove part (b) we now suppose that $b_g$ is proportional to $g|_{\goe_{-1}}$.
		Then \eqref{E:Aut.gr} restricts to an isomorphism
		\begin{equation}\label{E:Aut.gr.g}
			\Aut_\grr(\goe;g)=\OO\bigl(\goe_{-1};g|_{\goe_{-1}}\bigr)\cong\OO_2(\mathbb R).
		\end{equation}
		Hence, the orthogonal reflection in $\mathbb R^2$ mapping $(\lambda,\mu)$ to $-(\lambda,\mu)$ provides an orientation reversing $\phi\in\Aut_\grr(\goe;g)$ 
		such that the representations $\rho\circ\phi$ and $\rho$ are unitarily equivalent, cf.~\cite[\S3.2]{H23a}.
		In particular, $(\rho\circ\phi)(S)$ and $\rho(S)$ are isospectral. 
		On the other hand, $(\rho\circ\phi)(S)=\rho(-S)=-\rho(S)$ according to \eqref{E:rho.S.phi} as $S^G_{\phi\cdot g,\phi\cdot\mathfrak o}=S^G_{g,-\mathfrak o}=-S^G_{g,\mathfrak o}$.
		This shows that the spectrum of $\rho(S)$ is symmetrical about zero, whence $\eta_{\rho(S)}(s)=0$.
	\end{proof}

\section{The eta invariant of (2,3,5) nilmanifolds}\label{S:nilmf}

	In this section we provide proofs of Theorems~\ref{T:eta.nilmf.entire} and \ref{T:eta.nilmf.values} formulated in the introduction.
	Moreover, we show that the eta invariant of the odd signature operator on a Riemannian (2,3,5) nilmanifold vanishes, cf.~Proposition~\ref{P:APS235} below.

	Recall that $G$ denotes the simply connected nilpotent Lie group with Lie algebra $\goe$, cf.~\eqref{E:goe.intro} and \eqref{E:brackets.intro}. 
	Suppose $\Gamma$ is a lattice in $G$ and consider the closed nilmanifold $M=\Gamma{\setminus}G$.
	Recall that the left invariant (2,3,5) distribution $\mathcal D_G$ on $G$ spanned by $X_1$ and $X_2$ descends to a (2,3,5) distribution on the nilmanifold $M$ denoted by $\mathcal D_\Gamma$.
	Let $g$ be a graded Euclidean inner product on $\goe$ and recall that the corresponding left invariant fiberwise graded Euclidean inner product $g_G$ on $\mathfrak tG$ 
	descends to a fiberwise graded Euclidean inner product on $\mathfrak tM$ denote by $g_\Gamma$.
	Let $\mathfrak o$ be an orientation of $\goe$ and recall that the corresponding orientation $\mathfrak o_G$ of $G$ descends to an orientation on $M$ denoted by $\mathfrak o_\Gamma$.
	
	\begin{proof}[{Proof of Theorem~\ref{T:eta.nilmf.values}}]
		Suppose $\chi\colon\Gamma\to U(1)$ is a unitary character.
		The space of $L^2$-sections of the associated unitary flat line bundle $F_\chi=G\times_\chi\mathbb C$ 
		over $M=\Gamma{\setminus}G$ decomposes into a countable direct sum of irreducible unitary $G$-representations,
		\[
			\Gamma\bigl(F_\chi\bigr)=\bigoplus_\rho m(\rho)\cdot\rho.
		\]
		The multiplicities $m(\rho)$ appearing in this decomposition are known explicitly through a formula 
		due to Howe \cite[Thm~1]{H71} and Richardson \cite[Thm~4.5]{R71}, cf.~\cite[\S3]{H23b}.
		This gives rise to a decomposition of the operator $S=S^{\Gamma{\setminus}G,F_\chi}_{\mathcal D_\Gamma,g_\Gamma,\mathfrak o_\Gamma}$, 
		\[
			S=\bigoplus_\rho m(\rho)\cdot\rho(S^G_{g,\mathfrak o}),
		\]
		cf.~\eqref{E:S} and \eqref{E:SG.def}.
		The eta function decomposes accordingly, 
		\[
			\eta_S(s)=\sum_\rho m(\rho)\cdot\eta_{\rho(S^G_{g,\mathfrak o})}(s).
		\]
		Recall that $b_g$ is assumed to be proportional to $g|_{\goe_{-1}}$.
		Hence, in view of Lemma~\ref{L:eta.scalar} and Lemma~\ref{L:eta.generic}(b) only the Schr\"odinger representations contribute, i.e.,
		\begin{equation}\label{E:eta.deco}
			\eta_S(s)=\sum_\hbar m(\rho_\hbar)\cdot\eta_{\rho_\hbar(S^G_{g,\mathfrak o})}(s).
		\end{equation}

		If $\chi|_{\Gamma\cap C}$ is nontrivial, then no Schr\"odinger representation appears according to \cite[Lemma~5(II)]{H23b} 
		and we obtain $\eta_S(s)=0$, whence part (a) of the theorem.

		Assume from now on that $\chi|_{\Gamma\cap C}$ is trivial.
		Then the multiplicity of the Schr\"o\-din\-ger representation $\rho_\hbar$ is
		\begin{equation}\label{E:multi.rhoh}
			m(\rho_\hbar)=
			\begin{cases}
				|c+rn|&\text{if $\hbar=(c+rn)\frac{X_3/r}\gamma$ with $n\in\mathbb Z$, and}
				\\
				0&\text{otherwise.}
			\end{cases}
		\end{equation}
		This follows immediately from \cite[Lemma~5(II)]{H23b} if $\Gamma$ has the form considered there.
		In view of \cite[Lemma~2]{H23b} every lattice is of this form, up to a not necessarily graded automorphism $\phi\in\Aut(G)$.
		For general lattices the formula in \eqref{E:multi.rhoh} thus follows from the unitary equivalence of representations 
		$\rho_\hbar\circ\phi\sim\rho_{\hbar\cdot\det(\phi|_{\goe_{-1}})}$ which holds for all automorphisms, graded or not, cf.~\cite[\S3.2]{H23a}.
		
		Combining \eqref{E:eta.deco} with \eqref{E:multi.rhoh} and \eqref{E:etas} we obtain
		\begin{equation}\label{E:etaS.expl}
			\eta_S(s)
			=\left(\frac{2\pi/r}{\sqrt{g(\gamma,\gamma)}}\right)^{-s}
			\cdot\tilde\eta\bigl(s,\tfrac54\bigr)
			\cdot\sum_{\substack{n=-\infty\\c+rn\neq0}}^\infty\sign(c+rn)|c+rn|^{-s+1}.
		\end{equation}
		If $\chi|_{\Gamma\cap[G,G]}$ is trivial, then $c=0$ and the summands in \eqref{E:etaS.expl} cancel in pairs, whence part (b) of the theorem.
		
		To see part~(c) we now assume that $\chi|_{\Gamma\cap[G,G]}$ is nontrivial, hence $\frac cr\notin\mathbb Z$.
		Using \eqref{E:eta.Hurw} the formula in \eqref{E:etaS.expl} becomes
		\begin{equation}\label{E:eta.S.formula}
			\eta_S(s)
			=r
			\cdot\left(\frac{2\pi}{\sqrt{g(\gamma,\gamma)}}\right)^{-s}
			\cdot\tilde\eta\bigl(s,\tfrac54\bigr)
			\cdot\eta_\Hurw\bigl(s-1,\tfrac cr\bigr),
		\end{equation}
		whence \eqref{E:eta.S.formula.intro}.
		Combining this with \eqref{E:tetaa.odd} we obtain 
		\[
			\eta_S(-2l-1)=0,\qquad l=0,1,2,\dotsc
		\] 
		anew, cf.~\eqref{E:eta.S.odd}.
		Recall \cite[\href{https://dlmf.nist.gov/25.11.E14}{Eq.~(25.11.14)}]{NIST:DLMF} that 
		\[
			\zeta_\Hurw(-n,a)=-\frac{B_{n+1}(a)}{n+1},\qquad n=0,1,2,\dots
		\]
		and $B_n(1-a)=(-1)^nB_n(a)$ where $B_n(a)$ denotes the $n$-th Bernoulli polynomial \cite[\href{https://dlmf.nist.gov/24.2.E3}{Eq.~(24.2.3)}]{NIST:DLMF}. 
		In view of \eqref{E:eta.zeta.Hurw} we therefore have
		\begin{equation}\label{E:teta.Hurw.odd}
			\eta_\Hurw(-n,a)=0,\qquad n=1,3,5,\dotsc
		\end{equation}
		Combining this with \eqref{E:eta.S.formula} and \eqref{E:tetaa.zero} we obtain
		\[
			\eta_S(0)=0.
		\]
		Moreover, for $l=1,2,3,\dotsc$
		\begin{equation}\label{E:eta.S.2l.pre}
			\eta_S(-2l)
			=r\cdot\left(\frac{2\pi}{\sqrt{g(\gamma,\gamma)}}\right)^{2l}
			\cdot\res_{s=-2l}\tilde\eta\bigl(s,\tfrac54\bigr)
			\cdot\eta'_\Hurw\bigl(-2l-1,\tfrac cr\bigr).
		\end{equation}

		Using the relation between the polylogarithm $\Li_s(z)=\sum_{n=1}^\infty\frac{z^n}{n^s}$ and 
		the Hurwitz zeta function in \cite[\href{https://dlmf.nist.gov/25.12.E13}{Eq.~(25.12.13)}]{NIST:DLMF} we obtain from \eqref{E:eta.zeta.Hurw}
		\begin{equation}\label{E:eta.HUrw.Lis}
			\eta_\Hurw(1-s,a)=-2\mathbf i(2\pi)^{-s}\sin(\pi s/2)\Gamma(s)\left(\Li_s(e^{2\pi\mathbf ia})-\Li_s(e^{-2\pi\mathbf ia})\right).
		\end{equation}
		Differentiating at $s=2l+2$ and using $\overline{\Li_s(e^{2\pi\mathbf ia})}=\Li_{\bar s}(e^{-2\pi\mathbf ia})$ we get
		\begin{equation}\label{E:eta.dHurw.odd}
			\eta'_\Hurw(-2l-1,a)=(-1)^l(2\pi)^{-2l-1}(2l+1)!\cdot\Im\Li_{2l+2}(e^{2\pi\mathbf i a}).
		\end{equation}
		Combining this with \eqref{E:eta.S.2l.pre} and \eqref{E:tetaa.res} we obtain the formula in~\eqref{E:eta.S.2l}.
	\end{proof}

	\begin{proof}[{Proof of Theorem~\ref{T:eta.nilmf.entire}}]
		Recall that every unitary flat vector bundle $F$ over $\Gamma{\setminus}G$ is associated to a unitary representation $\chi\colon\Gamma\to U(N)$.
		Clearly, the eta function behaves additively with respect to this representation.
		Indeed, $F_{\chi_1\oplus\chi_2}=F_{\chi_1}\oplus F_{\chi_2}$ and thus
		\[
			S_{\mathcal D_\Gamma,g_\Gamma,\mathfrak o_\Gamma}^{\Gamma{\setminus}G,F_{\chi_1\oplus\chi_2}}
			=S_{\mathcal D_\Gamma,g_\Gamma,\mathfrak o_\Gamma}^{\Gamma{\setminus}G,F_{\chi_1}}
			\oplus S_{\mathcal D_\Gamma,g_\Gamma,\mathfrak o_\Gamma}^{\Gamma{\setminus}G,F_{\chi_2}}
		\]
		for any two representations $\chi_i\colon\Gamma\to U(N_i)$, cf.~\cite[Lemma~4.3(a)]{H22}.
		In particular,
		\[
			\eta_{F_{\chi_1\oplus\chi_2}}(\Gamma{\setminus}G,\mathcal D_\Gamma,\mathfrak o_\Gamma)
			=\eta_{F_{\chi_1}}(\Gamma{\setminus}G,\mathcal D_\Gamma,\mathfrak o_\Gamma)
			+\eta_{F_{\chi_2}}(\Gamma{\setminus}G,\mathcal D_\Gamma,\mathfrak o_\Gamma).
		\]
		W.l.o.g.\ we may therefore assume $\chi$ to be irreducible.
		In this case there exists a sublattice $\tilde\Gamma$ of finite index $N$ in $\Gamma$ 
		and a character $\tilde\chi\colon\tilde\Gamma\to U(1)$ such that $\chi$ is isomorphic to the representation induced from $\tilde\chi$, see \cite[Lemma~1]{B73}.
		In this case the Rumin complexes associated with $\chi$ and $\tilde\chi$ are essentially the same, whence
		\[
			S_{\mathcal D_\Gamma,g_\Gamma,\mathfrak o_\Gamma}^{\Gamma{\setminus}G,F_\chi}
			=S_{\mathcal D_{\tilde\Gamma},g_{\tilde\Gamma},\mathfrak o_{\tilde\Gamma}}^{\tilde\Gamma{\setminus}G,F_{\tilde\chi}}
		\]
		cf.~\cite[Lemma~4.3(b)]{H22}.
		In particular, 
		\[
			\eta_{F_\chi}(\Gamma{\setminus}G,\mathcal D_\Gamma,\mathfrak o_\Gamma)
			=\eta_{F_{\tilde\chi}}(\tilde\Gamma{\setminus}G,\mathcal D_{\tilde\Gamma},\mathfrak o_{\tilde\Gamma}).
		\]
		W.l.o.g.\ it thus suffices to consider unitary characters $\chi\colon\Gamma\to U(1)$, i.e., $N=1$.
		Recall from Theorem~\ref{T:eta.S.hinv} that the eta invariant can be computed using any graded Euclidean inner product.
		Hence, it remains to show $\eta_S(0)=0$ for a single graded Euclidean inner product $g$ on $\goe$ and all unitary characters $\chi\colon\Gamma\to U(1)$.
		This case, however, is covered by Theorem~\ref{T:eta.nilmf.values}.
	\end{proof}

	We close this section with the following Riemannian analogue.

	\begin{proposition}\label{P:APS235}
		The eta invariant of the odd signature operator on a (2,3,5) nilmanifold, with respect to any Riemannian metric and twisted by any unitary flat vector bundle, vanishes.
	\end{proposition}

	\begin{proof}
		On 5-dimensional manifolds the eta function of the untwisted odd signature operator vanishes identically, cf.~\cite[Remark~3a on p.~61]{APS75}.
		Hence, in this dimension, the eta invariant of the odd signature operator, twisted by any unitary flat vector bundle, 
		is independent of the Riemannian metric and coincides with the rho invariant, cf.~\cite[Thm~2.4]{APS75b}.

		Suppose $\Gamma\subseteq G$ is a lattice and let $M=\Gamma{\setminus}G$ denote the corresponding (2,3,5) nilmanifold.
		We have to show that the rho invariant associated with a unitary flat vector bundle $F$ over $M$ vanishes.
		Arguing as in the proof of Theorem~\ref{T:eta.nilmf.entire} above we see that 
		it suffices to consider unitary flat line bundles $F=F_\chi$ associated with unitary characters $\chi\colon\Gamma\to U(1)$.
		We will denote the rho invariant of $F_\chi$ by $\rho_\chi(M)$.
		Recall from \cite[\S2]{H23b} that the natural homomorphism $G\to G/[G,G]\cong\mathbb R^2$ gives rise to a short exact sequence
		\begin{equation}\label{E:ses.ab}
			0\to A\to\frac\Gamma{[\Gamma,\Gamma]}\to p(\Gamma)\to0
		\end{equation}
		where $A=\frac{\Gamma\cap[G,G]}{[\Gamma,\Gamma]}$ is a finite abelian group and $p(\Gamma)\cong\mathbb Z^2$.

		For every unitary character $\tau\colon\Gamma\to U(1)$ that factors through $p(\Gamma)$ we have
		\begin{equation}\label{E:rho.twist}
			\rho_{\chi\cdot\tau}(M)=\rho_\chi(M).
		\end{equation}
		To see this note first that $F_\tau$ is topologically trivial.
		Hence, $F_\chi$ and $F_{\chi\cdot\tau}=F_\chi\otimes F_\tau$ are topologically isomorphic.
		The relation between the rho invariant and the Chern--Simons invariants \cite{APS75b}, which we take from \cite[Thm~3.11.6]{G95}, reads
		\begin{equation}\label{E:int}
			\rho_{\chi\cdot\tau}(M)-\rho_\chi(M)
			=\pm\int_{SM}\Td(M)\ch(\Pi^+(\sigma))\Tch(\nabla^{\chi\cdot\tau},\nabla^\chi).
		\end{equation}
		As $M$ is parallelizable, the Todd class is trivial, $\Td(M)=1$.
		Using a Riemannian metric induced from a left invariant metric on $G$, we see that the bundle $\Pi^+(\sigma)$
		provided by the positive eigenspaces of the principal symbol \cite[\S3.8.2]{G95} is trivial, and so is its Chern character, $\ch(\Pi^+(\sigma))=1$.
		As $\rk(F_\chi)=1$ the transgressed Chern character associated with the flat connections on $F_\chi$ and $F_{\chi\cdot\tau}$ 
		is concentrated in degree one, $\Tch(\nabla^{\chi\cdot\tau},\nabla^\chi)\in\Omega^1(M)$, cf.~\cite[\S3.11.3]{G95}.
		We conclude that the integral in \eqref{E:int} vanishes, whence \eqref{E:rho.twist}.

		Since the sequence in~\eqref{E:ses.ab} splits there exists a homomorphism $\frac\Gamma{[\Gamma,\Gamma]}\to A$ that restricts to the identity on $A$.
		Composing this with the canonical projection we obtain a surjective homomorphism denoted by $\pi\colon\Gamma\to A$.
		In light of \eqref{E:rho.twist}	it suffices to show $\rho_{\alpha\circ\pi}(M)=0$ for all unitary characters $\alpha\colon A\to U(1)$.
		Note that $\tilde\Gamma:=\ker\pi$ is a normal sublattice in $\Gamma$ and let $\tilde M=\tilde\Gamma{\setminus}G$ denote the corresponding nilmanifold.
		The group $A=\Gamma/\tilde\Gamma$ acts freely (from the left) on $\tilde M$, turning the canonical projection $\tilde M\to M$ into an $A$-covering.
		In view of \cite[Thm~2.9]{APS75b} it suffices to show that Atiyah and Singer's \cite[\S7]{AS68} invariant $\sigma(g,\tilde M)$ vanishes for every nontrivial $g\in A$.
		To this end, let $Z\subseteq G$ be a central 1-parameter subgroup intersecting the lattice $\tilde\Gamma$ nontrivially.
		Then $S^1=\frac Z{\tilde\Gamma\cap Z}$ acts freely on $\tilde M$, turning the canonical projection $\tilde M\to N:=\tilde\Gamma{\setminus}G/Z$ 
		into an $A$-equivariant principal $S^1$-bundle.
		The associated disc bundle $W:=\tilde M\times_{S^1}D^2$ has boundary $\partial W=\tilde M$ and the $A$-action extends naturally from $\tilde M$ to $W$.
		Recall \cite[Thm~7.4]{AS68} that
		\[
			\sigma(g,\tilde M)=L(g,W)-\sign(g,W)
		\]
		where $\sign(g,W)$ denotes the $G$-signature of $W$, and $L(g,W)$ is a number obtained by integrating certain characteristic classes over the fixed point set of $g$ acting on $W$.
		We will complete the proof by showing that both summands on the right hand side vanish, for nontrivial $g\in A$.

		Let $e\in H^2(N;\mathbb R)$ denote the Euler class of the disc bundle $p\colon W\to N$.
		The natural map $H^3(W,\tilde M;\mathbb R)\to H^3(W;\mathbb R)$ fits into the commutative diagram:
		\begin{equation}\label{E:diag}
			\vcenter{
			\xymatrix{
				H^3(W,\tilde M;\mathbb R)\ar[r]\ar@{=}[d]_-{\textrm{Thom isom.}}&H^3(W;\mathbb R)\ar@{=}[d]^-{p^*}
				\\
				H^1(N;\mathbb R)\ar[r]^-{\cup e}&H^3(N;\mathbb R)
			}}
		\end{equation}
		Here the vertical isomorphism on the right is induced by the bundle projection $p$, 
		while the vertical isomorphism on the left is the Thom isomorphism of said disc bundle.
		We note that $N=\tilde\Gamma{\setminus}G/Z=\frac{\tilde\Gamma}{\tilde\Gamma\cap Z}{\setminus}\frac GZ$ is a 4-dimensional nilmanifold.
		By Nomizu's \cite{N54} theorem, $H^*(N;\mathbb R)=H^*(\noe;\mathbb R)$ where $\noe=\goe/\zoe$ is a nilpotent Lie algebra of Engel type, and $\zoe$ denotes the Lie algebra of $Z$.
		It is straight forward to compute the cohomology ring of $\noe$.
		This shows that the cup product $H^1(N;\mathbb R)\otimes H^2(N;\mathbb R)\to H^3(N;\mathbb R)$ vanishes.
		Hence, the upper horizontal map in~\eqref{E:diag} vanishes too.
		We conclude that $\sign(g,W)=0$, for all $g\in A$, cf.~\cite[\S2]{APS75b}.

		Consider a nontrivial $g\in A$ acting on $W$.
		If its fixed point set is nonempty then it coincides with $N$ considered as a submanifold of $W$ via the zero section of the disc bundle $W\to N$.
		In this case
		\[
			L(g,W)=\int_N\mathfrak P_g\bigl(p_1(N),e\bigr)
		\]
		where $\mathfrak P_g$ is a polynomial depending on $g$, see \cite[Thm~6.12]{AS68} or \cite[Thm~14.5]{LM89} for the general formula.
		The Pontryagin class $p_1(N)$ vanishes since $N$ is parallelizable.
		In order to prove $L(g,W)=0$ it thus suffices to show $e\cup e=0$ in $H^4(N;\mathbb R)$.
		By the exactness of the Gysin sequence $e$ generates the kernel of the map $H^2(N;\mathbb R)\to H^2(\tilde M;\mathbb R)$ induced by the bundle projection $\tilde M\to N$.
		Using Nomizu's theorem wee see that this kernel has dimension one, and $e\cup e=0$.
	\end{proof}


\section*{Funding information}
	This research was funded in whole or in part by the Austrian Science Fund (FWF) [10.55776/PAT8218224].
	For open access purposes, the author has applied a CC BY public copyright license to any author accepted manuscript version arising from this submission.

%

\end{document}